\documentclass{amsart}

\usepackage{color}
\usepackage{enumitem}
\usepackage{amsmath,amssymb,latexsym}
\usepackage{graphicx}
\usepackage[colorlinks,citecolor=blue,linkcolor=blue]{hyperref}
\usepackage[initials]{amsrefs}
\usepackage[nocompress,nosort]{cite}
\usepackage{mathtools}

\newtheorem{theorem}{Theorem}[section]
\newtheorem{lemma}[theorem]{Lemma}
\newtheorem{prop}[theorem]{Proposition}
\newtheorem{cor}[theorem]{Corollary}

\theoremstyle{definition}
\newtheorem{definition}[theorem]{Definition}

\theoremstyle{remark}
\newtheorem{remark}[theorem]{Remark}

\numberwithin{equation}{section}

\newcommand{\Rn}{\mathbb{R}^n}
\newcommand{\p}[1]{\partial {#1}}

\newcommand{\supp}{\operatorname{supp}}
\newcommand{\ang}[1]{\left\langle{#1}\right\rangle}
\newcommand{\compl}{{\mathsf{c}}}
\newcommand{\dx}{\mathop{}\!dx}
\newcommand{\dy}{\mathop{}\!dy}
\newcommand{\ds}{\mathop{}\!ds}
\begin{document}

\title[Large-time behavior of one-phase Stefan-type problems]{Large-time behavior of one-phase Stefan-type problems with anisotropic diffusion in periodic media}


%    author one information
\author{Norbert Po\v z\'ar}
\address[N. Po\v{z}\'ar]{Faculty of Mathematics and Physics, Institute of Science and Engineering, Kanazawa University, Kakuma, Kanazawa, 920-1192, Japan}
\curraddr{}
\email{npozar@se.kanazawa-u.ac.jp }
\thanks{}

\author{Giang Thi Thu Vu}
\address[G. T. T. Vu] {Department of Mathematics, Faculty of Information Technology, Vietnam National University of Agriculture, Hanoi, Vietnam}
\curraddr{}
\email{vtgiang@vnua.edu.vn}
\thanks{}


\date{\today}

\begin{abstract}
  We study the large-time behavior of solutions of a one-phase Stefan-type problem with anisotropic diffusion in periodic media on an exterior domain in a dimension $n \geq 3$. By a rescaling transformation that matches the expansion of the free boundary, we deduce the homogenization of the free boundary velocity together with the homogenization of the anisotropic operator. Moreover, we obtain the convergence of the rescaled solution to the solution of the homogenized Hele-Shaw-type problem with a point source and the convergence of the rescaled free boundary to a self-similar profile with respect to the Hausdorff distance.
\end{abstract}

\maketitle
\section{Introduction}

We analyze the behavior of an anisotropic one-phase Stefan-type problem with periodic coefficients on an exterior domain in a dimension $n \geq 3$. Our purpose is to investigate the asymptotic behavior of the solution of the problem \eqref{Stefan} below and its free boundary as time $t \to \infty$. The results in this paper are the generalizations of our previous work in \cite{PV} for the isotropic case.

We consider a compact set $K \subset \mathbb{R}^n$ that represents a source. We assume that $0 \in \operatorname{int}K$ and $K$  has a sufficiently regular boundary, $\partial K \in C^{1,1}$ for example. The one-phase Stefan-type problem with anisotropic diffusion is to find a  function $v(x,t): \mathbb{R}^n \times (0,\infty) \rightarrow [0,\infty)$ satisfying
\begin{equation}
	\label{Stefan}
	\left\{
	\begin{aligned}
		v_t- D_i(a_{ij}D_jv)&=0 && \text{ in } \{v>0\} \setminus K ,\\
		v &=1 && \text{ on } K,\\
		\frac{v_t}{|Dv|} &=g\,a_{ij}\,D_jv\, \nu_i && \text{ on }\partial \{v>0\},\\
		v(x,0)&=v_0 &&\text{ on } \mathbb{R}^n,
	\end{aligned}
	\right.
\end{equation}
where $D$ is the space gradient, $D_i$ is the partial derivative with respect to $x_i$, $v_t$ is the partial derivative of $v$ with respect to time variable $t$ and  $\nu=\nu(x,t)$ is the inward spatial unit normal vector of $\partial \{v>0\}$ at a point $(x,t)$. Here we use the Einstein summation convention.

The Stefan problem is a free boundary problem of parabolic type for phase transitions, typically describing the melting of ice in contact with a region of water. Here we consider the one-phase problem, where the temperature is assumed to be maintained at $0$ in one of the phases. We prescribe the Dirichlet boundary data $1$ on the fixed source $K$ and an initial temperature distribution $v_0$. Note that the results in this paper apply to a more general time-independent positive fixed boundary data, the constant function $1$ is taken only to simplify the notation. We also specify an inhomogeneous medium with the latent heat of phase transition $L(x)=\frac{1}{g(x)}$ and an anisotropic diffusion with the thermal conductivity coefficients given by $a_{ij}(x)$. The unknowns here are the temperature distribution $v$ and the phase interface $\partial \{v>0\}$, which is the so-called free boundary. Since the free boundary is a level set of $v$, the outward normal velocity of the moving interface is given by $\frac{v_t}{|Dv|}$. The free boundary condition thus says that the interface moves outward with the velocity $ga_{ij}D_jv \nu_i$ in the normal direction. Note that we can also rewrite the free boundary condition as
\begin{equation}
\label{Stefan bdr cond}
v_t=g\, a_{ij}\, D_jvD_i v.
\end{equation}

Throughout this paper, we will consider the problem under the following assumptions. The matrix $A(x)=(a_{ij}(x))$ is assumed to be symmetric, bounded, and uniformly elliptic, i.e., there exits some positive constants $\alpha$ and $\beta$ such that
\begin{equation}
\label{ellipticity}
\alpha |\xi|^2 \leq a_{ij}(x)\,\xi_i \xi_j \leq \beta|\xi|^2 \quad \mbox{ for all } x \in \mathbb{R}^n \mbox{ and } \xi \in \mathbb{R}^n.
\end{equation}

Moreover, we are interested in the problems with highly oscillating coefficients that guarantee an averaging behavior in the scaling limit, in particular,
\begin{equation}
\label{condition in media}
\begin{aligned}
  &\text{$a_{ij}$ and $g$ are $\mathbb{Z}^n$-periodic Lipschitz functions in $\mathbb{R}^n$},\\
&m\leq g \leq M \mbox{ for some positive constants } m \mbox{ and } M.
\end{aligned}
\end{equation}
	From the ellipticity \eqref{ellipticity} and the boundedness of $g$,  we also have
	\begin{equation}
	\label{ellipticity and Stefan bdr cond}
  m\alpha|\xi|^2\leq g(x)a_{ij}(x)\xi_i \xi_j \leq M\beta |\xi|^2 \mbox{ for all } x \in \Rn \mbox{ and } \xi \in \Rn.
	\end{equation}
	Furthermore, during almost the whole paper, the initial data is assumed to satisfy
\begin{equation}
\label{initial data}
\begin{aligned}
&v_0 \in C^2(\overline{\Omega_0 \setminus K}), v_0 > 0 \text{ in } \Omega_0, v_0=0 \mbox{ on } \Omega_0^\compl := \Rn \setminus
\Omega_0,
\mbox{ and } v_0=1 \mbox{ on } K,\\& |Dv_0| \neq 0 \text{ on } \partial \Omega_0,
\text{ for some bounded domain $\Omega_0 \supset K$.}
\end{aligned}
\end{equation}
Here we use a stronger regularity of the initial data than the general requirement to guarantee the well-posedness of the Stefan problem \eqref{Stefan} (see \cite{FK,K1}) and the coincidence of weak and viscosity solutions used in our work (see \cite{K3}). However, our convergence results rely on a crucial weak monotonicity \eqref{weak monotonicity} which holds provided the initial data satisfies \eqref{initial data}. Nevertheless, the asymptotic limit in Theorem \ref{th:main-convergence} is independent of the initial data. Therefore we are able to apply the results for more general initial data. In particular, it is sufficient if the initial data guarantees the existence of the (weak) solution satisfying the comparison principle, and we can approximate the initial data from below and from above by regular data satisfying \eqref{initial data}. As the classical problem \cite{PV}, $v_0 \in C(\Rn), v_0 = 1 \mbox{ on } K, v_0 \geq 0, \operatorname{supp} v_0$ compact is enough.

A global classical solution of the Stefan problem \eqref{Stefan} is not expected to exist due to the singularities of the free boundary, which might occur in finite time. This motivated the introduction of a generalized solution of this problem. In this paper, we will use the notions of weak solutions and viscosity solutions. The notion of weak solutions is more classical, defined by taking the integral in time of the classical solution $v$ and looking at the equation that the new function $u(x,t):= \int_{0}^{t}v(x,s)\ds$ satisfies. It turns out that if $v$ is sufficiently regular, then  $u(\cdot,t)$ solves the obstacle problem (see \cite{Baiocchi,Duvaut,FK,EJ,R1,R2,RReview})
\begin{equation}
\label{obstacle problem}
\left \{
\begin{aligned}
&u(\cdot,t) \in \mathcal{K}(t),\\
&\left(u_t - D_i(a_{ij}D_ju)\right)(\varphi - u) \geq f(\varphi -u) \mbox{ a.e } (x,t) \mbox{ for any } \varphi \in
\mathcal{K}(t),
\end{aligned}
\right.
\end{equation}
where $\mathcal{K}(t)$ is a suitable admissible functional space (see Section~\ref{sec:weak-sol}) and $f$ is
\begin{equation}
\label{f}
f(x)=
\left\{\begin{aligned}
&v_0(x), && v_0(x) > 0,\\
&-\displaystyle \frac{1}{g(x)}, && v_0(x) = 0.
\end{aligned}
\right.
\end{equation}
This formulation interprets the Stefan problem as a fixed domain problem and allows us to apply the well-known results in the general variational inequality theory. Indeed, the obstacle problem \eqref{obstacle problem} has a global unique solution $u$ for the initial data \eqref{initial data}. If the corresponding time derivative $v=u_t$ exists, it is called a \emph{weak solution} of the Stefan problem (\ref{Stefan}). Moreover, the homogenization of this problem was also studied using the homogenization theory of variational inequalities, see \cite{R3,K2,K3}. In a different approach based on the comparison principle structure, Kim introduced the notion of viscosity solutions of the Stefan problem as well as proved the global existence and uniqueness results in \cite{K1}, which were later generalized to the two-phase Stefan problem in \cite{KP}. The analysis of viscosity solutions relies on the comparison principle and pointwise arguments, which are more suitable for the study of the behavior of the free boundaries. The notions of weak and viscosity solutions were first introduced for the  classical homogeneous isotropic Stefan problem where $g(x)=1$ and the parabolic operator is the simple heat operator, however, it is natural to define the same notions for the Stefan problem \eqref{Stefan} and obtain the analogous results as observed in \cite{R3,K3}. Moreover, the notion of viscosity solutions is also applicable for more general, fully nonlinear parabolic operators and boundary velocity laws since it does not require the variational structure.  In \cite{K3}, Kim and Mellet showed that the weak and the viscosity solutions of \eqref{Stefan} coincide whenever the weak solution exists. We will use the strengths of both weak and viscosity solutions to study \eqref{Stefan}.

Among the first results on the asymptotic and large time behavior of solutions of the one-phase Stefan problem on an exterior domain was the work of Matano \cite{Matano} for the classical homogeneous isotropic Stefan problem in dimensions $n \geq 3$. He showed that in this setting, any weak solution eventually becomes classical after a finite time and that the shape of the free boundary approaches a sphere of radius $Ct^{\frac{1}{n}}$ as $t \to \infty$. Note that in our case of an inhomogeneous free boundary velocity we cannot expect the solution to become classical even for large times. In \cite{QV}, Quir\'{o}s and V\'{a}zques extended the results of \cite{Matano} to the case $n=2$ and showed the large-time convergence of the rescaled weak solution of the one-phase Stefan problem to the self-similar solution of the Hele-Shaw problem with a point source. The related Hele-Shaw problem is also called the quasi-stationary Stefan problem, where the heat operator is replaced by the Laplace operator. It typically models the flow of a viscous fluid injected in between two parallel plates which form the so-called Hele-Shaw cell, or the flow in porous media. The global stability of steady solutions of the classical Stefan problem on the full space was established in the recent work of Had\v zi\'c and Shkoller \cite{HS_CPAM,HS_Royal}, and further developed for the two-phase Stefan problem in \cite{HNS}.

The homogenization of the one-phase Stefan problem \eqref{Stefan} and the Hele-Shaw problem in both periodic and random media was obtained by Rodrigues in \cite{R3} and later by Kim and Mellet in \cite{K2,K3}. See also the work on the homogenization of the two-phase Stefan problem by Bossavit and Damlamian \cite{BD}. Dealing directly with the large-time behavior of the solutions on an exterior domain in inhomogeneous media, the work of the first author in \cite{P1}, and then the joint work of both authors in \cite{PV} showed the convergence of an appropriate rescaling of solutions of both models to the self-similar solution of the Hele-Shaw problem with a point source. The rescaled free boundary was shown to uniformly approach a sphere.

In this paper, we extend the previous results in \cite{PV} to the anisotropic case, where the heat operator is replaced by a more general linear parabolic operator of divergence form. This was indeed the setting considered in \cite{R3, K3} for the homogenization problems. In this setting, the variational structure is preserved, thus we are still able to use the notions of weak solutions as well as viscosity solutions and their coincidence. However, the main difficulties come from the loss of radially symmetric solutions which were used as barriers in the isotropic case and the homogenization problems appear not only for the velocity law but also for the elliptic operators. To overcome the first difficulty, we will construct barriers for our problem from the fundamental solution of the corresponding elliptic equation in a divergence form. Unfortunately, even though the unique fundamental solution of this elliptic equation exists for $n\geq 2$, its asymptotic behavior in dimension $n=2$ and dimensions $n\geq 3$ are significantly different. Moreover, we need to use the gradient estimate \eqref{gradient estimate fundamental sol} for the fundamental solution, which only holds for the periodic structure. Therefore, we will limit our consideration to periodic media and dimension $n\geq 3$. Following \cite{QV,P1,PV}, we use the rescaling of solutions as
\begin{equation*}
\begin{aligned}
  & v^\lambda(x,t):=\lambda^{\frac{n-2}{n}}v(\lambda^{\frac{1}{n}}x,\lambda t), && u^\lambda(x,t):=\lambda^{-\frac{2}{n}}u(\lambda^{\frac{1}{n}}x,\lambda t),
  \quad \lambda > 0.
\end{aligned}
\end{equation*}
Using this rescaling we can deduce the uniform convergence of the rescaled solution to a limit function away from the origin. In the limit, the fixed domain $K$ shrinks to the origin due to the rescaling, and the rescaled solutions develop a singularity at the origin as $\lambda \to \infty$. Moreover, in a periodic setting, the elliptic operator and velocity law should homogenize as $\lambda \to \infty$, and therefore  heuristically the limit function should be the self-similar solution of the Hele-Shaw-type problem with a point source
\begin{equation}
\label{hs-point-source}
\left\{
\begin{aligned}
-q_{ij}D_{ij}v&=C\delta && \mbox{ in }\{v>0\},\\
v_t&=\frac{1}{\left<\frac{1}{g}\right>}q_{ij}D_ivD_jv && \mbox{ on } \partial\{v>0\},\\
v(\cdot,0)&=0,
\end{aligned}
\right.
\end{equation}
where $\delta$ is the Dirac $\delta$-function, $(q_{ij})$ is a constant symmetric positive definite matrix depending only on $a_{ij}$, $C$ is a constant depending on $K, n, q_{ij}$ and the boundary data $1$, and the constant $\left<\frac{1}{g}\right>$ is the average value of the latent heat $L(x)=\frac{1}{g(x)}$. Similarly, the limit variational solution should satisfy the corresponding limit obstacle problem.

The first main result of this paper, Theorem \ref{convergence of variational solutions}, is the locally uniform convergence of the rescaled variational solution to the solution of the limit obstacle problem. Using the constructed barriers, we are able to prove that the limit function has the correct singularity as $|x| \to 0$. Moreover, the barriers also give the same growth rate for the free boundary as in the isotropic case. That is, the free boundary  expands with the rate $t^{\frac{1}{n}}$ when $t$ is large enough. The aim is then to prove the homogenization effects of the rescaling to our problem. The shrinking of the fixed domain $K$ in the rescaling also makes our current situation slightly different from the standard classical homogenization problem of variational inequalities, where the domain and the boundary condition are usually fixed. In addition, we also need to show that the rescaled parabolic operator becomes elliptic when $\lambda \to 0$. We will use the notion of the $\Gamma$-convergence introduced by De Giorgi and homogenization techniques developed by Dal Maso and Modica in \cite{DM1,DM2,D}. The issue here is that we need to modify the $\Gamma$-convergence sequence in order to use the integration by parts formula for the variational inequality. This will be done with the help of a cut-off function and the \emph{fundamental estimate}, Definition~\ref{def:fundamental-estimate}, for the $\Gamma$-convergence. Note that these techniques are applicable not only for the periodic case but also for the random case, thus we expect to extend our results to the problem in random media in the future.

 As the last step, we will use the coincidence of the weak and viscosity solution of the problem \eqref{Stefan} and the viscosity arguments to obtain the uniform convergence of the rescaled viscosity solution and its free boundary to the asymptotic profile in the second main result, Theorem \ref{convergence of rescaled viscosity solution}. Fortunately, all the viscosity arguments of the isotropic case can be adapted for the anisotropic case. Therefore the proof is similar to the proof of \cite[Theorem 4.2]{PV}, where we make use of a weak monotonicity \eqref{weak monotonicity} together with the comparison principle. An important point in the proof of\cite[Theorem 4.2]{PV} is that we need to apply Harnack's inequality for a parabolic equation which becomes elliptic in the limit. Here we can proceed as in
  the isotropic case since the rescaled elliptic operator does not change the constant in Harnack's inequality. As the arguments require only simple modifications, we will skip the proofs of some lemmas and refer to \cite{PV} for more details.

 In summary, we will show the following theorem.
	\begin{theorem}
		\label{th:main-convergence}
	The rescaled viscosity solution $v^\lambda$ of the Stefan-type problem \eqref{Stefan} converges
		locally uniformly to the unique self-similar solution $V$ of the Hele-Shaw type problem
		\eqref{hs-point-source} in $(\mathbb{R}^n \setminus \{0\}) \times [0,\infty)$ as $\lambda
    \rightarrow \infty$, where $(q_{ij})$ is a constant symmetric positive definite matrix depending only on $a_{ij}$, $C$ depends only on $q_{ij}, n$, the set $K$ and the boundary data $1$.
		Moreover, the rescaled free boundary $\partial \{(x,t):  v^\lambda(x,t) > 0\}$ converges to
		$\partial \{(x,t): V(x,t) > 0\}$ locally uniformly with respect to the Hausdorff distance.
	\end{theorem}

	As mentioned above, almost all of the arguments in our recent work hold for stationary ergodic random case. However, in this situation, we lose a very important pointwise gradient estimate \eqref{gradient estimate fundamental sol} for the fundamental solution of the corresponding elliptic equation to construct the barriers. In fact, for non-periodic coefficients, even though the optimal bounds for the gradient continue to hold for a bounded domain, they cannot hold in the large scale when $|x-y| \to \infty$. The results in \cite{MO,GM} tell us that for random stationary coefficients satisfying a logarithmic Sobolev inequality we have similar bounds for the gradient in local square average forms. This result cannot be upgraded to the pointwise bounds since there is no regularity to control the square average integral as in \cite[Remark 3.7]{GM}. However, it suggests the possibility to modify our approach to the random case. Another question is the extension of the present results to the dimension $n=2$. Since the unique (up to an addition of a constant) fundamental solution of the corresponding elliptic equation exists and the gradient estimates also hold in the two-dimensional case, we expect to obtain analogous results as in this paper. The essential reason that it remains open is the lack of a homogenization result for the fundamental solution (Green's function) in two dimensions, which is of an independent interest. This issue is under investigation by the authors.

	\subsection*{The structure of the paper is as follows:} The definitions and the well-known results for weak and viscosity solutions are recalled in Section 2. We also review some basic known facts about the fundamental solution of the corresponding elliptic equation. The rescaling is introduced and we discuss the convergence of the fundamental solution in the rescaling limit. The core of this section is the construction of a subsolution and a supersolution of the Stefan problem \eqref{Stefan} in Section \ref{sec: sub and super sol}. Moreover, we formulate the limit problems before giving the proofs of the main results in the later sections. Section 3 is our main contribution, where we prove the locally uniform convergence of the rescaled variational solutions. In Section 4, we deal with the locally uniform convergence of the rescaled viscosity solutions and their free boundaries.
\subsection*{Notation}
\label{sec:notation}
We will use the following notations throughout this paper.
For a set $A$, $A^\compl$  is its complement. Given a nonnegative function $v$, we will denote its positive set and free boundary as
$$\Omega(v):=\{(x,t): v(x,t)>0\}, \qquad  \Gamma(v):=\partial\Omega(v),$$
and for a fixed time $t$,
$$ \Omega_t(v):=\{x: v(x,t)>0\}, \qquad \Gamma_t(v):=\partial\Omega_t(v) .$$

We will denote the elliptic operator of divergence form and its rescaling as
\begin{equation}
  \label{operator-L}
  \mathcal{L}u=D_j(a_{ij}D_i u), \qquad \mathcal{L}^\lambda u=D_j(a_{ij}(\lambda^{\frac{1}{n}}x)D_i u), \quad \lambda > 0.
\end{equation}

We will also make use of the bilinear forms in $H^1(\Omega)$ and the inner product in $L^2(\Omega)$ as
\begin{equation*}
\begin{aligned}
a_{\Omega} (u,v)&=\int_{\Omega}{a_{ij}D_iuD_jv}\dx,  &&a^\lambda_{\Omega}(u,v)=\int_{\Omega}{a_{ij}(\lambda^{\frac{1}{n}}x)D_iuD_jv}\dx, \quad \lambda >0\\
q_{\Omega} (u,v)&=\int_{\Omega}{q_{ij}D_iuD_jv}\dx, &&\langle u,v \rangle _\Omega = \int_{\Omega}uv\dx,
\end{aligned}
\end{equation*}
where $q_{ij}$ are the constant coefficients of the homogenized operator. We omit the set $\Omega$  in the notation if $\Omega = \mathbb{R}^n$.

\section{Preliminaries}
\subsection{Notion of solutions}
\subsubsection{Weak solutions}
\label{sec:weak-sol}
As for the classical one-phase Stefan problem, we will define the weak solutions of \eqref{Stefan} using the corresponding variational problem given in \cite{FK,K3}.
Let $B=B_R(0)$, $D=B \setminus K$ for some fixed $R \gg 1$.  Find $u \in L^2(0,T; H^2(D))$ such that $u_t \in L^2(0,T;L^2(D))$ and
\begin{equation}
\label{Variational problem}
\left\{\begin{aligned}
u(\cdot,t) &\in \mathcal{K}(t), && 0 < t < T,\\
(u_t-\mathcal{L}u)(\varphi -u) &\geq f(\varphi -u), &&\text{ a.e } (x,t) \in D \times (0,T) \text{ for any } \varphi \in \mathcal{K}(t),\\
u(x,0)&=0 \text{ in } D,\\
\end{aligned}\right.
\end{equation}
where the admissible set $\mathcal{K}(t)$ is
$$\mathcal{K}(t)=\{\varphi \in H^1(D),\varphi \geq 0, \varphi =0 \text{ on } \partial B, \varphi =t \text{ on } K \}$$
and
\begin{equation*}
f(x):= \left \{
\begin{aligned}
& v_0(x) && \text{for } x\in \Omega_0,\\
&-\frac{1}{g(x)} && \text{for } x \in \Omega_0^\compl.
\end{aligned}
\right.
\end{equation*}
We use the standard notation $H^k$ and $W^{k,p}$ for Sobolev spaces.

Following \cite{FK}, if $v(x,t)$ is a classical solution of the Stefan problem  (\ref{Stefan}) in $D \times (0,T)$ and $R$ is sufficiently large depending on $T$, then the function $u(x,t):=\int_{0}^{t}v(x,s)\ds$ solves \eqref{Variational problem}.
On the other hand, it was shown in \cite{FK,R1} that the variational problem \eqref{Variational problem} is well-posed for initial data $v_0$ satisfying \eqref{initial data}.
\begin{theorem}[Existence and uniqueness for the variational problem]
	If $v_0$ satisfies (\ref{initial data}), then the problem (\ref{Variational problem}) has a unique solution satisfying
	$$
	\begin{aligned}
	u &\in L^\infty(0,T; W^{2,p}(D)), \qquad 1\leq p \leq \infty,\\
	u_t &\in L^\infty(D \times (0,T)),\\
	\end{aligned}
	$$
	and
	$$
	\left\{\begin{aligned}
	u_t-\mathcal{L} u &\geq f, &&u \geq 0 ,\\
	u(u_t- \mathcal{L} u-f)&=0
	\end{aligned}\right.
	\mbox{ a.e in } D\times (0, \infty).
	$$
\end{theorem}
Thus we define the \emph{weak solution} of the Stefan problem (\ref{Stefan}) as the time derivative $u_t$ of the solution $u$ of (\ref{Variational problem}). Note that as in \cite[Lemma~3.6]{K3}, the solution does not depend on the choice of $R$ if $R$ is sufficiently large.

We list here some useful properties of the weak solutions for later use, see \cite{FK,R1,K3}.
\begin{prop}
	\label{boundedness of weak solution}
	The unique solution $u$ of (\ref{Variational problem}) satisfies $u_t \in C(\overline D \times [0, T])$ and
	$$0 \leq u_t \leq C \qquad \mbox{ a.e. } D \times (0,T),$$
	where $C$ is a constant depending on $f$. In particular, $u$ is Lipschitz with respect to $t$ and
	$u$ is $C^{\alpha}(D)$ with respect to $x$ for all $\alpha \in (0,1)$. Furthermore, if $0 \leq t
	<s \leq T$, then $u(\cdot,t) < u(\cdot,s)$ in $\Omega_s(u)$ and also
	$\Omega_0 \subset \Omega_t(u) \subset \Omega_s(u)$.
\end{prop}
\begin{proof}
Let us only give a remark on $u_t \in C(\overline D \times [0, T])$, the rest is standard following the arguments in the cited papers and the elliptic regularity theory.  The regularity of weak solutions and their free boundaries was studied by many authors, see \cite{C,CF,FN, Di-B, Ziemer, Sacks}. If $a_{ij}=\delta_{ij}$, the temperature $u_t$ is continuous in $\Rn \times [0,\infty)$ due to the result of Caffarelli and Friedman \cite{CF}. By a change of coordinates, the continuity of $u_t$ can also be obtained when the coefficients are constants. Using a different approach for more general singular parabolic equations, Di Benedetto \cite{Di-B}, Ziemer \cite{Ziemer} and Sacks \cite{Sacks} showed that the continuity also holds in the case $a_{ij}=a_{ij}(x)$ satisfying \eqref{ellipticity}. Note that the assumptions on the $L^\infty$-bound of the weak solution $v=u_t$ and the $L^2$-bounds of its derivatives as in \cite{Ziemer, Di-B, Sacks} are guaranteed by Proposition~\ref{boundedness of weak solution} above and \cite[Theorem 3]{FK} or \cite[Corollary 2, Theorem 4]{F}.
\end{proof}

\begin{lemma}[Comparison principle for weak solutions]
  Suppose that $f \leq \hat{f}$. Let $u, \hat{u}$ be solutions of (\ref{Variational problem}) for respective $f, \hat{f}$. Then $u \leq \hat{u}$ and
	$$v \equiv u_t \leq \hat{u}_t \equiv \hat{v}.$$
\end{lemma}
\subsubsection{Viscosity solutions}
\label{sec: viscosity sol}
Generalized solutions of the Stefan problem \eqref{Stefan} can also be defined via the comparison principle, leading to the viscosity solutions introduced in \cite{K1}. In the following, $Q$ is the space-time cylinder $Q:=(\mathbb{R}^n \setminus K) \times [0,\infty)$.
\begin{definition}
	\label{def of viscos.subsol}
	A nonnegative upper semicontinuous function $v=
	v(x,t)$ defined in $Q$ is a viscosity subsolution of (\ref{Stefan}) if:
	\begin{enumerate}[label=\alph*)]
		\item \label{continuous expanding} For all $T \in (0,\infty)$, $\overline{\Omega (v)} \cap \{t \leq T\} \cap Q \subset \overline{\Omega (v) \cap \{t < T\}}.$
		\item For every $\phi \in C^{2,1}_{x,t}(Q)$ such that $v-\phi$ has a local maximum in $\overline{\Omega(v)} \cap \{t\leq t_0\} \cap Q$ at $(x_0,t_0)$, the following holds:
		\begin{enumerate}[label=\roman*)]
			\item If $v(x_0,t_0)>0$, then $(\phi_t- \mathcal{L} \phi )(x_0,t_0) \leq 0$.
			\item If $(x_0,t_0) \in \Gamma(v), |D\phi (x_0,t_0)| \neq 0$ and $(\phi_t-\mathcal{L} \phi )(x_0,t_0)>0$, then
			\begin{equation}
			(\phi_t-ga_{ij}D_j\phi \nu_i |D\phi|)(x_0,t_0) \leq 0,
			\end{equation}
			where $\nu$ is inward spatial unit normal vector of $\p{\{v>0\}}$.
		\end{enumerate}
	\end{enumerate}
		Analogously, a nonnegative lower semicontinuous function $v(x,t)$ defined in $Q$ is a viscosity
	supersolution if (b) holds with maximum replaced by minimum, and with inequalities reversed in
	the tests for $\phi$ in (i--ii). We do not need to require (a).
\end{definition}
\begin{remark}
	As in \cite[Remark~2.4]{Pozar15}, the condition \ref{continuous expanding} guarantees the continuous expansion of the support of the subsolution $v$, which prevents ``bubbles'' closing up, that is, it prevents $v$ becoming instantly positive in the whole space or in an open set surrounded by a positive phase.
\end{remark}

\begin{definition}
	\label{def vis.subsol with data}
	A viscosity subsolution of (\ref{Stefan}) in $Q$ is a viscosity subsolution of (\ref{Stefan}) in $Q$ with initial data $v_0$ and boundary data $1$ if:
	\begin{enumerate}[label=\alph*)]
		\item $v$ is upper semicontinuous in $\bar Q, v=v_0$ at $t=0$ and $v \leq 1$ on $\Gamma$,
		\item $\overline{\Omega(v)}\cap \{t=0\}=\overline{\{x: v_0(x)>0\}} \times \{0\}$.
	\end{enumerate}
	A viscosity supersolution is defined analogously by requiring (a) with $v$ lower semicontinuous
	and $v \geq 1$ on $\Gamma$. We do not need to require (b).
\end{definition}

A viscosity solution is both a subsolution and a supersolution:
\begin{definition}
	\label{def vis.sol with data}
	The function $v(x,t)$ is a viscosity solution of (\ref{Stefan}) in $Q$ (with initial data $v_0$ and boundary data $1$) if $v$ is a viscosity supersolution and $v^\star$ is a viscosity subsolution of (\ref{Stefan}) in $Q$ (with initial data $v_0$ and boundary data $1$).
	Here $v^\star$ is the upper semicontinuous envelopes of  $v$ define by
	\begin{align*}
	w^\star(x,t) := \limsup_{(y,s) \rightarrow (x,t)}w(y,s).
	\end{align*}
\end{definition}
The notion of viscosity solutions of the classical Stefan problem was first introduced in \cite{K1}. It was generalized to the problem \eqref{Stefan} in \cite{K3} including a
comparison principle for ``strictly separated'' initial data. More importantly,  in \cite{K3} the authors proved the coincidence of weak
and viscosity solutions which will be used as a crucial tool in our work.
\begin{theorem}[cf. {\cite[Theorem 3.1]{K3}}]
	Assume that $v_0$ satisfies (\ref{initial data}). Let $u(x,t)$ be the unique solution of (\ref{Variational problem}) in $B \times [0,T]$ and let $v(x,t)$ be the solution of
	\begin{equation}
	\label{coincidence eq}
	\left\{
	\begin{aligned}
	v_t-\mathcal{L} v &=0 && \mbox{in } \Omega(u) \setminus K,\\
	v &=0 && \mbox{on } \Gamma(u),\\
	v &=1 &&\mbox{in } K,\\
	v(x,0) &=v_0(x).
	\end{aligned}
	\right.
	\end{equation}
	Then $v(x,t)$ is a viscosity solution of (\ref{Stefan}) in $B \times [0,T]$ with initial data
	$v(x,0)=v_0(x)$, and
	$u(x,t)= \int_{0}^{t}v(x,s)\ds$.
\end{theorem}
By the coincidence of weak and viscosity solutions, we have a more general comparison principle as follows.
\begin{lemma}[cf. {\cite[Corollary 3.12]{K3}}]
	Let $v^1$ and $v^2$  be, respectively, a viscosity subsolution and supersolution of the Stefan problem (\ref{Stefan}) with continuous initial data $v_0^1 \leq v_0^2$ and boundary data $1$. In addition, suppose that $v_0^1$ (or $v_0^2$) satisfies condition (\ref{initial data}). Then
	$v^1_\star \leq v^2  \mbox{ and }  v^1 \leq (v^2)^\star \mbox{ in }  \mathbb{R}^n \setminus K
	\times [0,\infty).$
\end{lemma}
\begin{remark}
We first note that a classical subsolution (supersolution) of (\ref{Stefan}) is also a viscosity subsolution (supersolution) of \eqref{Stefan} in $Q$ with initial data $v_0$ and boundary data $1$ by standard arguments.

Moreover, if $\Omega(u)$ is not smooth, we need to understand the solution of (\ref{coincidence eq}) as the one given by Perron's method as
	\begin{equation*}
	v=\sup \{w\mid w_t- \Delta w \leq 0 \mbox{ in } \Omega(u), w\leq 0 \mbox{ on } \Gamma(u), w\leq 1 \mbox{ in } K, w(x,0) \leq v_0(x)\},
	\end{equation*}
	which allows $v$ to be discontinuous on $\Gamma(u)$.
\end{remark}
\subsection{The fundamental solution of a linear elliptic equation}
In this section, we will recall some important facts about the fundamental solution of the self-adjoint uniformly elliptic second order linear equation of divergence form
\begin{equation}
\label{elliptic eq}
\begin{aligned}
&-\mathcal{L}u=0,
\end{aligned}
\end{equation}
in dimension $n \geq 3$, where $\mathcal{L}$ was defined in \eqref{operator-L} and $a_{ij}(x)$ satisfy \eqref{ellipticity} and \eqref{condition in media}. This fundamental solution will be used to construct barriers for the Stefan problem \eqref{Stefan}.

We define the fundamental solution of \eqref{elliptic eq} as Green's function in the whole space following \cite{LSW, ABL}.
\begin{definition}
	We say that $G: \Rn \times \Rn \to \mathbb{R}$ is the fundamental solution (Green's function) of \eqref{elliptic eq} if $G(\cdot,y)$ is the weak (distributional) solution of $-\mathcal{L}G(\cdot,y)= \delta_y$, where $\delta_y$ is the Dirac measure at $y$, i.e.,
	\begin{equation*}
	\begin{aligned}
	&\int_{\Rn}a_{ij}D_jG(\cdot,y) D_i\varphi \dx= \varphi(y), && \forall y \in \Rn, &&&\forall \varphi \in C^\infty_0(\Rn),
	\end{aligned}
	\end{equation*}
	and $\lim_{|x-y|\to \infty}G(x,y)=0$.
\end{definition}
The existence and uniqueness of the fundamental solution were given by the remark following \cite[Corollary 7.1]{LSW} or more precisely by \cite[Theorem 1]{ABL}.
\begin{theorem}[cf. {\cite[Theorem 1]{ABL}}]
	Assume that $n \geq 3$, $a_{ij}(x)$ satisfy \eqref{ellipticity} and  \eqref{condition in media}. Then, there exists a unique fundamental solution $G$ of \eqref{elliptic eq} such that $G(\cdot, y) \in H^1_{loc}(\Rn \setminus \{y\}) \cap W^{1,p}_{loc}(\Rn), p <\frac{n}{n-1},$ and for some constant $C>0$ we have
	\begin{equation}
		\label{bound for fund.sol n>2}
	C^{-1}|x-y|^{2-n}\leq G(x,y) \leq C|x-y|^{2-n}, \quad \forall x,y \in \Rn.
	\end{equation}
\end{theorem}
\begin{remark}
	Note that in any bounded domain $U$ of $\Rn \setminus \{0\}$, $G(\cdot,y)$ satisfies all the properties of a weak solution of a uniformly elliptic equation. The fundamental solution of \eqref{elliptic eq} also has the following properties (for more details, see \cite{LSW,M, GT }):
	\begin{itemize}
		\item $G(x,y)=G(y,x)$
		\item $G(\cdot,y) \in C^{1,\alpha}(U)$ for some $\alpha >0$.
		\item The function $u(x)= \int_{\Rn}G(x,y)f(y)\dy$
		is a weak solution in $H^1_{\text{loc}}(\Rn)$ of the equation $-\mathcal{L}u = f$
		for any $f \in C^\infty_0(\Rn)$.
		\item When the coefficients $a_{ij}$ are constants, the fundamental solution can be given explicitly as
		\begin{equation}
		\label{fundamental const coef}
    G^0(x,y):= \frac{1}{(n-2)\alpha_n\sqrt{\operatorname{det} A}}\left(\sum_{ij}(A^{-1})_{ij}(x_i-y_i)(x_j-y_j)\right)^{\frac{2-n}{n}},
		\end{equation}
    where $(A^{-1})_{ij}$ are the elements of the inverse matrix of $(a_{ij})$, $\operatorname{det} A$ is the determinant of $(a_{ij})$ and $\alpha_n$ is the volume of the unit ball in $\Rn$.
	\end{itemize}
\end{remark}

Moreover, in a periodic setting, the results in \cite[Proposition 5]{ABL} gives the bounds on the gradient of the fundamental solution.
\begin{lemma}[cf. {\cite[Proposition 5]{ABL}}]
	If $n\geq 2$ and $A$ is periodic then the fundamental solution $G$ of \eqref{elliptic eq} satisfies the following gradient estimates:
	\begin{align}
	\label{gradient estimate fundamental sol}
	&\exists C>0, && \forall x\in \Rn, &&\forall y\in \Rn,&&|D_xG(x,y)| \leq \frac{C}{|x-y|^{n-1}},\\
	&\exists C>0, && \forall x\in \Rn, &&\forall y\in \Rn,&&|D_yG(x,y)| \leq \frac{C}{|x-y|^{n-1}}.
	\end{align}
\end{lemma}
Using the technique of $G$-convergence, the authors in \cite{ZKOH} established the results on the homogenization and the asymptotic behavior of the fundamental solution of \eqref{elliptic eq}. We refer to \cite{ZKOH,D} for the definition of $G$-convergence and more details of the homogenization problems.
\begin{lemma}[cf. {\cite[Chapter III, Theorem 2]{ZKOH}}]
	\label{Asymptotic expansion fund.sol}
	Let $n \geq 3$, $A$ satisfy \eqref{ellipticity},  \eqref{condition in media} and $G^\varepsilon$ be the fundamental solution of
		\begin{equation}
		\label{elliptic epsilon eq}
		\begin{aligned}
		&-\mathcal{L}^\varepsilon u :=-D_i \left(a_{ij} \left(\frac{x}{\varepsilon} \right) D_j u\right)=0.
		\end{aligned}
		\end{equation} Then $G^\varepsilon$ converges locally uniformly to $G^0$ in $\mathbb{R}^{2n} \setminus \{x=y\}$ as $\varepsilon \to 0$, where  $G^0$ is the fundamental solution of
	\begin{equation}
	\label{elliptic hom eq}
	\begin{aligned}
	&-\mathcal{L}^0 u:= -q_{ij}D_{ij}u=0,
	\end{aligned}
	\end{equation}
  and $(q_{ij})$ is a constant symmetric positive definite matrix depending only on $a_{ij}$.  Moreover, if we denote $G$ as the fundamental solution of \eqref{elliptic eq}, then we will have the asymptotic expression
	\begin{equation}
	\label{asymptotic fund.sol}
	G(x,y)=G_0(x,y)+|x-y|^{2-n}\theta(x,y),
	\end{equation}
	where $\theta(x,y) \to 0$ as $|x-y| \to \infty$ uniformly on the set $\{|x|+|y|<a|x-y|\}$, $a$ is any fixed positive constant.
\end{lemma}
\subsection{Rescaling}
\label{sec: rescaling}
Recall that we use the notations $\mathcal{L}, \mathcal{L}^\lambda$ for the operators as defined in \eqref{operator-L}.
Following \cite{QV,P1, PV}, for $\lambda >0$ and $n \geq 3$ we rescale the solution $v$ of the problem \eqref{Stefan} as
$$v^\lambda(x,t)=\lambda^{\frac{n-2}{n}}v(\lambda^{\frac{1}{n}}x,\lambda t).$$
Clearly $v^\lambda$ is a solution of
\begin{equation}
\label{rescaled equation}
\left\{
\begin{aligned}
\lambda^{\frac{2-n}{n}}v_t^\lambda- \mathcal{L}^\lambda v^\lambda &=0 && \mbox{ in } \Omega(v^\lambda) \setminus K^\lambda,\\
v^\lambda &= \lambda^{\frac{n-2}{n}} && \mbox{ on } K^\lambda,\\
\frac{v_t^\lambda}{|D v^\lambda|} &=g^\lambda(x) a^\lambda_{ij}(x) D_j v^\lambda \nu_i && \mbox{ on } \Gamma(v^\lambda),\\
v^\lambda(\cdot,0)&=v_0^\lambda && \mbox{ in } \Rn,
\end{aligned}
\right.
\end{equation}
where $K^\lambda := K / \lambda^{\frac{1}{n}}, \Omega_0^\lambda:= \Omega_0/\lambda^{\frac{1}{n}}$, $g^\lambda(x)=g(\lambda^{\frac{1}{n}}x), a_{ij}^\lambda(x)=a_{ij}(\lambda^{\frac{1}{n}}x)$ and $v_0^\lambda(x)= \lambda^{\frac{n-2}{n}}v_0(\lambda^{\frac{1}{n}}x)$.

The corresponding rescaling of the weak solution $u$ of the variational problem \eqref{Variational problem} can be shown to be (see  \cite{QV,P1, PV})
$$u^\lambda(x,t)= \lambda^{-\frac{2}{n}} u(\lambda^{\frac{1}{n}}x, \lambda t),$$
which solves the rescaled obstacle problem
\begin{equation}
\label{rescaled Variational problem}
\left\{
\begin{aligned}
u^\lambda(\cdot,t) &\in \mathcal{K}^\lambda(t), && 0<t<\infty,\\
(\lambda^{\frac{2-n}{n}}u^\lambda_t-\mathcal{L}^\lambda u^\lambda)(\varphi -u^\lambda) &\geq f^\lambda(x)(\varphi -u^\lambda) &&\text{a.e. } (x,t) \in \Rn \times (0,\infty)\\
&&& \text{for any } \varphi \in \mathcal{K}^\lambda(t),\\
u^\lambda(x,0)&=0,
\end{aligned}
\right.
\end{equation}
where
$\mathcal{K^\lambda}(t)=\{\varphi \in H^1(\mathbb{R}^n),\varphi \geq 0, \varphi =\lambda^{\frac{n-2}{n}}t \text{ on } K^\lambda \}$ and $f^\lambda(x)=f(\lambda^{\frac{1}{n}}x)$.
\begin{remark}
	\label{ig compact support}
  The admissible set $\mathcal{K}^{\lambda}(t)$ can be defined in this way due to \cite[Remark~2.13]{PV}. Note that for any fixed time $t$, the admissible set $\mathcal{K}^\lambda(t)$ depends on $\lambda$.
\end{remark}

\subsubsection{Convergence of the rescaled fundamental solution}
By Lemma~\ref{Asymptotic expansion fund.sol}, we have the following convergence result on the rescaled fundamental solution.
\begin{lemma}
	\label{rescaled fund.sol}
	Let $G$ be the fundamental solution of \eqref{elliptic eq} in dimension $n \geq 3$ and $G^\lambda$ be its rescaling as
	\begin{equation*}
	G^\lambda(x,y)=\lambda^{\frac{n-2}{n}}G(\lambda^{\frac{1}{n}}x,\lambda^{\frac{1}{n}}y).
	\end{equation*}
	Then $G^\lambda$ is the fundamental solution of
		\begin{equation}
		\label{elliptic rescaled eq}
		-\mathcal{L}^\lambda u = 0,
		\end{equation}
		 and $|G^\lambda(x,y) - G^0(x,y)| \to 0$
 uniformly on every compact subset of $\mathbb{R}^{2n} \setminus \{(x,x)\in \mathbb{R}^{2n}\}$ where $G^0$ is the fundamental solution of \eqref{elliptic hom eq}.
\end{lemma}
\begin{proof}
	We will show that $G^\lambda$ is the fundamental solution of \eqref{elliptic rescaled eq}, then the result follows directly from Lemma~\ref{Asymptotic expansion fund.sol} with $\varepsilon = \lambda^{-\frac{1}{n}}$.

		For simplicity, we will check that $G^\lambda$ satisfies the definition of the fundamental solution of \eqref{elliptic rescaled eq} for fixed $y=0, F(x)= G(x,0)$ and $F^\lambda(x):=\lambda^{\frac{n-2}{n}}F(\lambda^{\frac{1}{n}}x)$.

		Indeed, we have $D_jF^{\lambda}(x)={\lambda}^{\frac{n-1}{n}}D_jF({\lambda}^{\frac{1}{n}}x)$. Take a function  $\varphi \in C^\infty_0(\Rn)$, then
		\begin{align*}
		\int_{\Rn}a_{ij}^\lambda(x)D_jF^{\lambda}(x)D_i\varphi(x)\dx&=\int_{\Rn}{\lambda}^{\frac{n-1}{n}}a_{ij}({\lambda}^{\frac{1}{n}}x)D_jF({\lambda}^{\frac{1}{n}}x)D_i\varphi(x)\dx\\
		&=\int_{\Rn}{\lambda^{-\frac{1}{n}}} a_{ij}(y)D_jF(y)D_i\varphi(\lambda^{-\frac{1}{n}}y)\dy\\
		&=\int_{\Rn} a_{ij}(y)D_jF(y)D_i\tilde{\varphi}(y)\dy\\
		&=\tilde{\varphi}(0) = \varphi(0),
		\end{align*}
		where $\tilde{\varphi}(y)=\varphi(\lambda^{-\frac{1}{n}}y)$.
		Moreover, $F^\lambda$ satisfy the estimate \eqref{bound for fund.sol n>2} since $F$ has this property. Hence, by definition, $F^\lambda$ is the fundamental solution of \eqref{elliptic rescaled eq}.
	\end{proof}
	\begin{remark}
		The rate of this convergence as well as the rate of convergence for derivatives were also derived in \cite{ALin}.
	\end{remark}
\subsection{Construction of barriers from a fundamental solution}
\label{sec: sub and super sol}
The main goal of this section is to construct a subsolution and a supersolution of \eqref{Stefan}
 from a fundamental solution of the elliptic equation \eqref{elliptic eq} so that we can use them as barriers to track the behavior of the support of a solution of \eqref{Stefan}.

From now on, we will  let $\mathcal{L}^0$ be the limit of the operators of $\mathcal{L}^\lambda$ as in Lemma~\ref{Asymptotic expansion fund.sol} and consider the fundamental solutions of \eqref{elliptic eq}, \eqref{elliptic rescaled eq}  and \eqref{elliptic hom eq} with a pole at the origin as
\begin{equation*}
\begin{aligned}
F(x)&:= G(x,0), && F^\lambda(x):=G^\lambda(x,0)= \lambda^{\frac{n-2}{n}}F(\lambda^{\frac{1}{n}}x),
&&& F^0(x)&:= G^0(x,0).
\end{aligned}
\end{equation*}
	Note that $F^0$ is preserved under the rescaling by \eqref{fundamental const coef}.
\subsubsection{Construction of a supersolution}
\label{supersol}
Define
\begin{equation*}
\theta(x,t):={[C_1F(x)-C_2 t^{\frac{2-n}{n}}]_+},
\end{equation*}
where $[s]_+ := \max(s, 0)$ denotes the positive part of $s$
and $C_1,C_2$ are non-negative constants to be chosen later.
It easily follows that in $\{\theta>0\} \setminus \{x = 0\}$,
\begin{align*}
\theta_t(x,t)&= \frac{C_2(n-2)}{n}t^{\frac{2-2n}{n}} \geq 0,\\
D\theta&= C_1DF,\\
\mathcal{L}\theta&=0,\\
\theta_t -\mathcal{L}\theta &\geq 0.
\end{align*}
Due to the estimates \eqref{bound for fund.sol n>2} and \eqref{gradient estimate fundamental sol}, there exists a constant $C$ such that
\begin{equation}
\label{estimate of F(x)}
\begin{aligned}
C^{-1}|x|^{2-n} &\leq F(x) \leq C|x|^{2-n},\\
|DF(x)|&\leq C|x|^{1-n} .
\end{aligned}
\end{equation}
Then for $(x,t) \in \p{\{\theta>0\}}$ we have
\begin{equation*}
 C_2 t^{\frac{2-n}{n}}=C_1F(x) \geq C_1C^{-1}|x|^{2-n},
\end{equation*}
which yields
\begin{equation*}
t^{\frac{1}{n}} \leq \left(\frac{C_1}{CC_2}\right)^{\frac{1}{2-n}}|x|.
\end{equation*}
Thus on $\partial \{\theta>0\}$,
\begin{equation*}
\theta_t(x,t) = \frac{C_2(n-2)}{n}t^{\frac{2-2n}{n}} \geq \frac{n-2}{n}\left(\frac{C_1}{C}\right)^{\frac{2-2n}{2-n}} C_2^{\frac{n}{2-n}}|x|^{2-2n}.
\end{equation*}
Fix any $t_0 >0$. We can choose $C_1$ large enough and $C_2$ small enough such that
\begin{align*}
\label{supersol condition}
&\theta_t(x,t) \geq M\beta C_1^2 C^2 |x|^{2-2n} \geq M\beta|D\theta(x,t)|^2 \mbox{ on } \partial \{\theta>0 \},\\
&\theta> 1 \mbox{ on } K \mbox{ and }\theta(x,t_0)>v(x,t_0),
\end{align*}
 where $\alpha, \beta$ are the elliptic constants from \eqref{ellipticity}. By \eqref{ellipticity and Stefan bdr cond}, $\theta_t \geq g\,a_{ij}\, D_j\theta D_i\theta$ on $\partial\{\theta>0\}$ and by \eqref{Stefan bdr cond}, $\theta$ is a supersolution of \eqref{Stefan} in $\Rn \times [t_0, \infty)$.

\subsubsection{Construction of a subsolution}
\label{subsol}
Let $h$ be the function constructed in \cite[Appendix A]{K3} with $\mathcal{L}h=n, Dh(x)=(A(x))^{-1}x$ and let $c, \tilde{c} >0$ be constants such that
\begin{equation}
\label{quadratic growhth}
c|x|^2 \leq h(x) \leq \tilde{c}|x|^2.
\end{equation}

Consider the function
\begin{equation}
\label{subsol formula}
\theta(x,t):=\left[c_1F(x)+\frac{c_2h(x)}{t}-c_3 t^{\frac{2-n}{n}}\right]_+ \chi_{E}(x,t)
\end{equation}
with non-negative constants $c_1,c_2,c_3$ to be chosen later,
where
 \begin{equation*}
 \begin{aligned}
   E&:={\{(x,t): \tfrac{\partial F_b}{\partial r}(|x|,t)<0, t > 0\}}, &&F_b(r,t):=Cc_1r^{2-n}+\frac{c_2\tilde{c}r^2}{t}-c_3t^{\frac{2-n}{n}},
 \end{aligned}
 \end{equation*}
 $C,\tilde{c}$ are constants as in \eqref{estimate of F(x)}, \eqref{quadratic growhth}.
We claim that we can choose constants $c_1,c_2,c_3,t_0$ such that $\theta$ is a subsolution of \eqref{Stefan} for $ t \in [t_0,\infty)$.
The differentiation of  $\theta$ on the set $\{\theta>0\} \setminus \{x = 0\}$ yields
\begin{equation}
\label{subsolution computations}
\begin{aligned}
  D\theta(x,t)&=  c_1DF(x)+ \frac{c_2A(x)^{-1}x}{t},\\
\mathcal{L}\theta(x,t)&=\frac{c_2n}{t},\\
\theta_t(x,t)&=-\frac{c_2h(x)}{t^2}+\frac{c_3(n-2)}{n}t^{\frac{2-2n}{n}}=t^{\frac{2-2n}{n}}\left[\frac{c_3(n-2)}{n}-\frac{c_2h(x)}{t^\frac{2}{n}}\right],
\end{aligned}
\end{equation}
and thus
\begin{equation*}
  \theta_t(x,t) -\mathcal{L}\theta(x,t) = t^{\frac{2-2n}{n}}\left[\frac{c_3(n-2)}{n}-\frac{c_2h(x)}{t^{\frac{2}{n}}}-\frac{c_2n}{t^{\frac{2-n}{n}}}\right] < 0 \quad \text{for } t \gg 1.
\end{equation*}
Thus, we can choose $t_0$ large enough such that $\theta_t -\mathcal{L}\theta < 0$ for $t\geq t_0$.

 Now we will prove the continuity of $\theta$. We have
 \begin{equation}
 \label{stronger barrier}
 0\leq \theta(x,t) \leq [F_b(|x|,t)]_+ \chi_{E}(x,t)=:F_b^+(x,t),
 \end{equation}
 and hence $\Omega_t(\theta) \subset \Omega_t(F_b^+)$ for all $t$.
 We see that
\begin{equation}
\label{radius of E(t)}
\frac{\partial F_b}{\partial r}(r,t)= Cc_1(2-n)r^{1-n}+\frac{2c_2\tilde{c}r}{t}<0 \Leftrightarrow r < \left(\frac{Cc_1(n-2)}{2c_2\tilde{c}}\right)^{\frac{1}{n}}t^{\frac{1}{n}}=:r_0(t)
\end{equation}
and hence $E=\{(x, t): |x|< r_0(t), t >0\}$. Clearly $\theta$ is continuous in the set $\{\theta>0\} \setminus \{x = 0\}$. Furthermore, $\theta$ is continuous in $E\setminus \{x = 0\}$ and $\theta=0$ on $E^\compl$.  We will show that we can choose the constants such that $\theta$ is continuous through boundary of $E$. Indeed, for $(x_0,t) \in \partial E$, $t > 0$,
\begin{equation*}
F_b(|x_0|,t)= F_b(r_0(t),t)=C_{F_b} t^{\frac{2-n}{n}},
\end{equation*}
where $C_{F_b}=(Cc_1)^{\frac{2}{n}}(c_2\tilde{c})^{\frac{n-2}{n}}\left[\left(\frac{n-2}{2}\right)^{\frac{2}{n}} \frac{n}{n-2}\right]-c_3$ . We can choose $c_1,c_2,c_3$ such that $C_{F_b}<0$. Then $F_b(|x_0|,t)<0$ for all $(x_0, t) \in \p{E}$, $t>0$. Since $(x,t) \mapsto F_b(|x|, t)$ is continuous in a neighborhood of $\partial E$, we deduce by \eqref{stronger barrier} that $\theta=0$ in a neighborhood of $\partial E$ and therefore it is continuous across $\partial E$. Note that $C_{F_b} <0$ if and only if
\begin{equation}
\label{subsolution coef condition 1}
c_3 \geq C_0 (c_1)^{\frac{2}{n}}(c_2)^{\frac{n-2}{n}},
\end{equation}
 where $C_0$ is a constant depending only on $n,C,\tilde{c}$.

We finally need to show that we can choose suitable constants such that $\theta$ satisfies the subsolution condition on the free boundary.

We first note that $\theta(x,t) \geq \tilde{\theta}(x,t):=\left[Cc_1|x|^{2-n}-c_3 t^{\frac{2-n}{n}}\right]_+ $. Then $\Omega(\tilde{\theta}) \subset \Omega(\theta)$, or more precisely, there exists a constant $\tilde{C}$ such that
\begin{equation}
\label{subsolution bdr condition 1}
|x|\geq \tilde{C}t^{\frac{1}{n}} \mbox{ for all } (x,t) \in \partial \{\theta>0\}.
\end{equation}

By \eqref{subsolution computations} we have
\begin{align*}
\theta_t(x,t)&\leq c_3 t^{\frac{2-2n}{n}},\\
|D\theta(x,t)|^2&=c_1^2|DF(x)|^2+ \frac{2c_1c_2}{t}DF(x)\cdot A^{-1}x+\frac{c_2^2}{t^2}|A^{-1}x|^2,\\
&\geq  \frac{2c_1c_2}{t}DF(x)\cdot A^{-1}x+\frac{c_2^2}{t^2}|A^{-1}x|^2.
\end{align*}

Since $A$ is a symmetric bounded matrix satisfying the ellipticity \eqref{ellipticity}, then these properties also hold for $A^{-1}$ and $A^{-2}$ with appropriate constants. Hence, for $(x,t) \in \p{\{\theta>0\}}$,
\begin{equation*}
\begin{aligned}
|D\theta(x,t)|^2 &\geq \frac{c_2^2}{t^2}\tilde{\alpha}|x|^2 -  \frac{2c_1c_2}{t}C_A|DF(x)||x| &&\mbox{ for some } \tilde{\alpha}, C_A >0\\
& \geq \frac{c_2^2}{t^2}\tilde{\alpha}|x|^2 -  \frac{2c_1c_2}{t}CC_A|x|^{2-n} && (\mbox{by \eqref{estimate of F(x)}})\\
&\geq \left(c_2^2 \tilde{\alpha}\tilde{C}^2 -2c_1c_2CC_A\tilde{C}^{2-n}\right)t^{\frac{2-2n}{n}} &&(\mbox{by \eqref{subsolution bdr condition 1}}).
\end{aligned}
\end{equation*}
We want to choose $c_1,c_2,c_3$ such that $\theta_t \leq m\alpha |D\theta|^2$ on $\partial\{\theta>0\}$, which will hold if
\begin{equation}
\label{subsolution coef condition 2}
c_3 \leq m\alpha \left(c_2^2 \tilde{\alpha}\tilde{C}^2 -2c_1c_2CC_A\tilde{C}^{2-n}\right)=:C_0^1 c_2^2 - C_0^2 c_1c_2,
\end{equation}
where $C_0^1,C_0^2$ are fixed positive constants.  Then by \eqref{elliptic epsilon eq}, $\theta_t \leq g\,a_{ij}\,D_j\theta D_i \theta$ on $\partial \{\theta>0\}$.

The conditions \eqref{subsolution coef condition 1} and \eqref{subsolution coef condition 2} hold if we choose some suitable $c_1,c_2,c_3$. For example, fix any $c_1>0$, choose $c_2$ large enough such that
$$ C_0 (c_1)^{\frac{2}{n}}(c_2)^{\frac{n-2}{n}}<C_0^1 c_2^2 - C_0^2 c_1c_2 .$$
Note that the above inequality holds for $c_2$ large enough since for a fixed $c_1>0$, the right hand side tends to $\infty$ as $c_2 \to \infty$ faster than the left hand side. Then \eqref{subsolution coef condition 1} and \eqref{subsolution coef condition 2} hold for any $c_3$ which is between these two numbers. Fix $t_0$ such that $\theta_t - \mathcal{L}\theta < 0$ in $\{\theta>0\}$ for chosen $c_2,c_3$ and $t \geq t_0$. Choosing a smaller $c_1$ if it is needed, we can assume that the support of $\theta(\cdot,t_0)$ is contained in $\Omega_{t_0}(v), \theta(x,t_0) \leq v(x,t_0)$ and $\theta<1$ on $\partial K$. Thus, with the help of \eqref{Stefan bdr cond}, we see that $\theta$ is a subsolution of the Stefan problem \eqref{Stefan} for that choice of constants.
 \subsubsection{Some results on the barriers for the Stefan problem \eqref{Stefan}}
 Due to the construction above, we can use the functions of the form
 \begin{equation}
 \label{barrier form}
 \theta(x,t):= [C_1F(x)-C_2t^{\frac{2-n}{n}}]_+
\end{equation}
with $C_1,C_2 >0$ as barriers for the Stefan problem \eqref{Stefan}. Since our purpose is to study the asymptotic behavior, we first observe the convergence of the rescaled barriers.
\begin{lemma}
	\label{barrier convergence of rescaled }
	Let $\theta$ be a function of the form \eqref{barrier form} and $\theta^\lambda:=\lambda^{\frac{n-2}{n}}\theta(\lambda^{\frac{1}{n}}x, \lambda t)$. Then $\theta^\lambda \to \theta^0$ locally uniformly in $(\Rn \setminus \{0\}) \times [0,\infty)$, where \begin{equation}
	\theta^0(x,t):= [C_1 F^0(x) -C_2t^{\frac{2-n}{n}}]_+.
	\end{equation}
\end{lemma}
\begin{proof}
	We have
	\begin{equation*}
	\theta^\lambda(x,t)= [C_1 F^\lambda(x)-C_2 t^{\frac{2-n}{n}}]_+,
	\end{equation*}
	where $F^\lambda(x)=\lambda^{\frac{n-2}{n}}F(\lambda^{\frac{1}{n}}x)$ .
	By Lemma \ref{rescaled fund.sol},  $F^\lambda \to F^0$ locally uniformly in $\Rn \setminus \{0\}$ and the lemma follows.
	\end{proof}
	Moreover we will also need to know the integral of the barriers in time to analyze the weak solution of the Stefan problem \eqref{Stefan}.
	\begin{lemma}
		\label{barrier integration in time lemma}
		Let $\Theta(x,t):= \int_0^t \theta(x,s)\ds$. Then $\Theta(x,t)$ has the form
		\begin{equation}
		\label{barrier integration in time}
		\begin{aligned}
      \Theta(x,t)=C_1F(x)t-\frac{C_2n}{2}t^{\frac{2}{n}}+o(F(x)), \mbox{ as } |x| \to 0.
		\end{aligned}
		\end{equation}
	\end{lemma}
	\begin{proof}
		We can derive \eqref{barrier integration in time} simply by integrating the function $\theta$ of the form \eqref{barrier form}. Since $\theta$ has the form \eqref{barrier form}, we see that
		\begin{equation*}
		\begin{aligned}
		&\begin{aligned}
		& \theta > 0 &&\mbox{ if } t> s(x),\\
		& \theta = 0 && \mbox{ if } t \leq s(x),
		\end{aligned}
		&& \mbox{ where } s(x)=\left(\frac{C_1}{C_2}F(x)\right)^{\frac{n}{2-n}}.
		\end{aligned}
		\end{equation*}
		Thus,
		\begin{equation*}
		\Theta(x,t)= \left\{
		\begin{aligned}
		& 0, &&t \leq s(x),\\
		&\int_{s(x)}^t{(C_1F(x)-C_2s^{\frac{2-n}{n}})\ds}, && t> s(x).
		\end{aligned}
		\right.
		\end{equation*}
		When $t> s(x)$,
		\begin{equation*}
		\begin{aligned}
		\Theta(x,t)&=C_1F(x)t-\frac{C_2n}{2}t^{\frac{2}{n}}-C_1F(x)s(x)+\frac{C_2n}{2}(s(x))^{\frac{2}{n}}\\
		&=C_1F(x)t-\frac{C_2n}{2}t^{\frac{2}{n}}+\frac{n-2}{2}\frac{(C_1)^{\frac{2}{2-n}}}{(C_2)^{\frac{n}{2-n}}}(F(x))^{\frac{2}{2-n}}\\
		&= C_1F(x)t-\frac{C_2n}{2}t^{\frac{2}{n}}+ C(F(x))^{\frac{2}{2-n}}.
		\end{aligned}
		\end{equation*}
		Since $F(x)$ has a singularity at $x=0$ by \eqref{bound for fund.sol n>2} then $s(x) \to 0$ and  $C(F(x))^{\frac{2}{2-n}}=o(F(x))$ as $|x|\to 0$, which completes the proof.
		\end{proof}
From these barriers, we can obtain the rate of expansion of the support for viscosity solutions.
	\begin{lemma}
		\label{viscos free boundary bound}
		Let $n \geq 3$ and $v$ be a viscosity solution of (\ref{Stefan}). There exists $t_0>0$ and constants $C, C_1, C_2>0$ such that for $t\geq t_0$,
		\begin{align*}
		C_1t^{\frac{1}{n}}\leq \min_{\Gamma_t(v)}|x|\leq \max_{\Gamma_t(v)}|x|\leq C_2t^{\frac{1}{n}}
		\end{align*}
		and for $0 \leq t\leq t_0$,
		$$\max_{\Gamma_t(v)}|x|\leq C_2.$$
		Moreover,
		$$0 \leq v(x,t) \leq C|x|^{2-n}.$$
	\end{lemma}
	\begin{proof}
    We deduce the bound for $v(x,t)$ first. Let $F(x)$ be the fundamental solution of the elliptic equation \eqref{elliptic eq} as in Section \ref{sec: sub and super sol}. Then $\hat{\theta}=CF(x)$ is a stationary solution of the equation $v_t- \mathcal{L}v
    =0$. Its integral in time is also a solution of the variational inequality problem with $\hat{f}=CF(x)$. If we take $C$ large enough then $\hat{f}\geq f$ and $\hat{\theta}\geq 1$ on $K$. Applying the comparison principle for the variational problem, \cite[Proposition 2.2]{QV}, we have $v(x,t) \leq CF(x) \leq \tilde{C}|x|^{2-n}$ by \eqref{bound for fund.sol n>2}.

		The bound on the support of $v(\cdot,t)$ at all times has been proved in \cite[Lemma~3.6]{K3}.

Now consider $\theta_1, \theta_2$ that are respectively a subsolution and a supersolution of the Stefan problem \eqref{Stefan} for $t\geq t_0$ as constructed in Section \ref{supersol} and \ref{subsol}.	The bounds on the support of $v$ for $t \geq t_0$ follow directly from the behavior of the supports of $\theta_1, \theta_2$.
	\end{proof}
\subsection{Limit problems}
The expected limit problem is the corresponding Hele-Shaw type problem with a point source.
\subsubsection{Limit problem for $v^\lambda$}
We expect $v^\lambda$ to converge to a solution of
\begin{equation}
\label{Hele_shaw point source problem}
\left\{
\begin{aligned}
\mathcal{L}^0 v &=0 &&\mbox{ in } \{v>0\},\\
\frac{v_t}{|Dv|}&= \displaystyle \frac1L q_{ij}D_jv \nu_i &&\mbox{ on } \partial \{v>0\},\\
\lim_{|x| \rightarrow 0}\displaystyle \frac{v}{F^0}&=C,\\
v(x,0)&=0 &&\mbox{ in } \mathbb{R}^n \setminus \{0\},
\end{aligned}
\right.
\end{equation}
where $C,L$ are positive constants, $q_{ij}$ are constants of the operator $\mathcal{L}^0
$ and $F^0$ is the fundamental solution of \eqref{elliptic hom eq}.

Since $Q:=(q_{ij})$ is symmetric and positive definite, we can write $Q=P^2$, where $P$ is also a symmetric positive definite matrix. Let $\tilde{v}(x,t):=v(Px,t)$. A direct computation then shows that the problem \eqref{Hele_shaw point source problem} becomes the classical Hele-Shaw problem with a point source for function $\tilde{v}$,
\begin{equation}
\label{Hele_shaw point source problem classical}
\left\{
\begin{aligned}
\Delta \tilde{v} &=0 &&\mbox{ in } \{\tilde{v}>0\},\\
\tilde{v}_t&= \frac1L |D\tilde{v}|^2 &&\mbox{ on } \partial \{v>0\},\\
\lim_{|x| \rightarrow 0} \frac{\tilde{v}}{|x|^{2-n}}&=C,\\
\tilde{v}(x,0)&=0 &&\mbox{ in } \mathbb{R}^n \setminus \{0\}.
\end{aligned}
\right.
\end{equation}
The problem \eqref{Hele_shaw point source problem classical} has a unique classical solution $\tilde{V}$ which is given explicitly (see \cite{P1,PV},  for instance). Thus \eqref{Hele_shaw point source problem} has unique classical solution $V(x,t):= \tilde{V}(P^{-1}x,t)$, which is continuous in $(\Rn \setminus \{0\}) \times [0,\infty)$.
\subsubsection{Limit problem for $u^\lambda$}
\label{sec: limit variational problem}
Suppose that  $V=V_{C,L}$ is the classical solution of \eqref{Hele_shaw point source problem} above and set
\begin{equation}
\label{limit function U}
U(x,t):= \int_0^t V(x,s)\ds.
\end{equation}
It is known that the time integral of the solution of the classical Hele-Shaw problem with a point source \eqref{Hele_shaw point source problem classical} satisfies an obstacle problem derived in \cite{P1}. Following \cite{P1} and using a change of variables again,  we see that $U$ uniquely solves the following problem, which is our limit variational problem:
\begin{equation}
\label{Limit problem for variational solution}
\left\{
\begin{aligned}
w &\in \mathcal{K}_t,\\
q(w,\phi) &\geq \left<-L,\phi\right >,&& \forall \phi \in W_1,\\
q(w, \psi w) &= \left<-L,\psi w\right>, &&\forall \psi \in W_2,
\end{aligned}
\right.
\end{equation}
where
$$\mathcal{K}_t= \left \{ \varphi \in \bigcap_{\varepsilon>0}H^1(\mathbb{R}^n \setminus B_\varepsilon) \cap C(\mathbb{R}^n \setminus B_\varepsilon): \varphi \geq 0, \lim_{|x| \rightarrow 0}\frac{\varphi(x)}{F^0(x)}=Ct\right \},$$
\begin{align}
\label{set V}
W_1&=\left\{ \phi \in H^1(\mathbb{R}^n \setminus B_\varepsilon):\phi \geq 0, \phi = 0 \mbox{ on } B_\varepsilon \mbox{ for some } \varepsilon >0\right\},\\
\label{set W}
W_2&=W_1 \cap C^1(\mathbb{R}^n).
\end{align}
\subsubsection{Near-field limit}
		Using the boundedness provided by Lemma~\ref{viscos free boundary bound}, we have the following general near-field limit result similar to \cite{QV}.

		\begin{theorem}[Near-field limit]
			\label{Near field limit Theorem} The viscosity solution $v$ of the Stefan problem (\ref{Stefan}) converges to the unique solution $P = P(x)$ of the exterior Dirichlet problem
			\begin{equation}
			\label{near filed limit}
			\left\{
			\begin{aligned}
			\mathcal{L} P&=0, && x \in \mathbb{R}^n \setminus  K,\\
			P&=1, && x \in K,\\
			\lim_{|x| \rightarrow \infty}P(x)&=0,
			\end{aligned}
			\right.
			\end{equation}
			as $t \rightarrow \infty$ uniformly on compact subsets of $\overline{K^\compl}$.
			\end{theorem}
			\begin{proof}
		Follow the arguments in proof of \cite[Lemma~8.4]{QV} and note that by Lemma~\ref{viscos free boundary bound} the support of $v$ expands to the whole space as time $t \to \infty$.
				\end{proof}

	The results on the isolated singularity of solutions of linear elliptic equations in \cite{SW} allow us to deduce the asymptotic behavior of $P$ as $|x| \to \infty$.
				\begin{lemma}
					\label{C*}
					There exists a constant $C_*=C_*(K, n)$ such that the solution $P$ of the problem (\ref{near filed limit}) satisfies
          $$\lim_{|x| \rightarrow \infty}\frac{P(x)}{F(x)}=C_*,$$
					where $F(x)$ is the fundamental solution of the elliptic equation $-\mathcal{L}v=0$ in $\Rn$.
					\end{lemma}
					\begin{proof}
				Lemma~\ref{C*} is a direct corollary of \cite[Theorem 5]{SW}. The arguments follow the same techniques as in \cite[Lemma~4.3]{QV} using a general Kelvin transform and Green's function for linear elliptic equations. Following \cite[Lemma~4.3]{QV}, it can also be shown that the constant $C_*$ depends continuously on the data of the fixed boundary $\Gamma=\p{K}$.
						\end{proof}

\section{Uniform convergence of the rescaled variational solutions}
The purpose of this section is to show the first main result on the uniform convergence of the rescaled variational solutions, which is similar to \cite[Theorem 3.1]{PV}.
\begin{theorem}
	\label{convergence of variational solutions}
	Let $u$ be the unique solution of variational problem (\ref{Variational problem}) and
	$u^\lambda$ be its rescaling. Let $U_{A,L}$ be the unique solution of limit problem \eqref{Limit problem for variational solution} where $A=C_*$ as in Lemma~\ref{C*}, and $L=\left<\frac{1}{g}\right>$ as in Lemma~\ref{media}. Then the functions $u^\lambda$ converge locally uniformly to $U_{A,L}$ as $\lambda \rightarrow \infty$ on $\left(\mathbb{R}^n \setminus \{0\}\right) \times \left[0,\infty\right)$.
\end{theorem}
The classical homogenization results of variational inequalities are usually stated for a fixed bounded domain. Since our admissible set $\mathcal{K}^\lambda(t)$ defined in Section~\ref{sec: rescaling} changes with $\lambda$, we will need to refine the proof. We will use the techniques of the $\Gamma$-convergence introduced in \cite{D} and \cite{K3}. Note that these techniques can be applied not only for the periodic case but also for stationary ergodic coefficients over a probability space $(A,\mathcal{F},P)$.
\subsection{The averaging property of media and the \texorpdfstring{$\Gamma$}{Gamma}-convergence}
We recall the following lemma on the averaging property of periodic media, which also holds for more general stationary ergodic media.
\begin{lemma}[cf. {\cite[Section 4, Lemma~7]{K2}, see also \cite{P1}}]
	\label{media}
	For a given $g$ satisfying  \eqref{condition in media}, there exists a constant, denoted by
	$\left<\frac{1}{g}\right>$, such that if $\Omega \subset \mathbb{R}^n$ is a bounded measurable set and if
	$\{u^\varepsilon\}_{\varepsilon>0} \subset L^2(\Omega)$ is a family of functions such that $u^\varepsilon \rightarrow u$ strongly in $L^2(\Omega)$ as $\varepsilon \rightarrow 0$, then
	\begin{equation*}
	\lim_{\varepsilon\rightarrow 0} \int_{\Omega}\frac{1}{g\left(\frac{x}{\varepsilon}\right)}u^\varepsilon(x)\dx
	=\int_{\Omega}\left<\frac{1}{g}\right>u(x)\dx.
	\end{equation*}
	The quantity $\ang{\frac{1}{g}}$ in periodic setting is the average of $\frac{1}{g}$ over one period.
\end{lemma}

We also need some basic concepts and results of the $\Gamma$-convergence which are taken from \cite{D}.
Let $\Omega$ be a bounded open set in $\Rn$. Consider the functional
\begin{equation}
\label{functionals in Gamma convergence}
J^\lambda(u,\Omega):=\left\{
\begin{aligned}
&\int_{\Omega}a_{ij}(\lambda^{\frac{1}{n}}x)D_iuD_ju\dx &&\mbox{ if } u \in H^1(\Omega),\\
&\infty &&\mbox{ otherwise}.
\end{aligned}
\right.
\end{equation}

\begin{definition}[cf. {\cite[Proposition 8.1]{D}}]
	Let X be a metric space. A sequence of functionals $F_h$ is said to $\Gamma(X)$-converge to $F$ if the following conditions are satisfied:\begin{enumerate}[label=(\roman*)]
		\item For every $u \in X$ and for every sequence $(u_h)$ converging to $u$ in $X$, we have
		\begin{equation*}
		F(u) \leq \liminf_{h \rightarrow 0}F_h(u_h).
		\end{equation*}
		\item For every $u \in X$, there exists a sequence $(u_h)$ converging to $u$ in $X$, such that
		\begin{equation*}
		F(u)=\lim_{h\rightarrow 0}F_h(u_h).
		\end{equation*}
	\end{enumerate}
\end{definition}
It is known that the $\Gamma(L^2)$-convergence of $J^\lambda$ is equivalent to the $G$-convergence of elliptic operator $\mathcal{L}^\lambda$ (see \cite[Theorem 22.4]{D} and \cite[Theorem 4.3]{K3}) and we have a crucial result on Gamma-convergence of $J^\lambda$ as follows.
\begin{theorem}[cf. {\cite[Theorem 4.3]{K3}}]
	\label{gamma convergence}
	The functionals $J^\lambda$ $\Gamma(L^2)$-converge  as $\lambda \rightarrow \infty$ to a functional $J^0$, where $J^0$ is a quadratic functional of the form
	\begin{equation*}
	J^0(u):=\left\{
	\begin{aligned}
	&\int_{\Omega}q_{ij}D_iuD_ju\dx &&\mbox{ if } u \in H^1(\Omega),\\
	&\infty &&\mbox{ otherwise}.
	\end{aligned}
	\right.
	\end{equation*}
	Here the constants $q_{ij}$ are the coefficients of the limit operator $\mathcal{L}^0$ as in Lemma~\ref{Asymptotic expansion fund.sol}.
\end{theorem}

To deal with the Dirichlet boundary condition, we need to use cut-off functions and the \emph{fundamental estimate} below. Here we denote as $\mathcal{A}$ the class of all open subsets of $\Omega$.
	\begin{definition}\cite[Definition 18.1]{D}
		 Let $A',A''\in \mathcal{A}$ with $A' \Subset A''$. We say that a function $\varphi: \Rn \rightarrow \mathbb{R}$ is a \textit{cut-off function} between $A'$ and $A''$ if $\varphi \in C^\infty_0(A''), 0 \leq \varphi \leq 1$ on $\Rn$, and $\varphi =1$ in a neighborhood of $\overline{A'}$	.
	\end{definition}
  \begin{definition}\label{def:fundamental-estimate}\cite[Definition 18.2]{D}
		Let $F: L^p(\Omega) \times \mathcal{A} \rightarrow [0,\infty]$ be a non-negative functional. We say that $F$ satisfies the \textit{fundamental estimate} if for every $\varepsilon>0$ and for every $A',A'',B \in \mathcal{A}$, with $A' \Subset A''$, there exists a constant $M>0$ with the following property: for very $u,v \in L^p(\Omega)$, there exists a cut-off function $\varphi$ between $A'$ and $A''$, such that
		\begin{equation}
		\label{cut-off.est}
		\begin{aligned}
			F(\varphi u + (1-\varphi)v, A' \cup B) &\leq (1+\varepsilon)(F(u,A'')+F(v,B))\\
			&+ \varepsilon(\|u\|^p_{L^p(S)}+\|v\|^p_{L^p(S)}+1)+ M\|u-v\|^p_{L^p(S)},
		\end{aligned}
		\end{equation}
		where $S=(A''\setminus A') \cap B$. Moreover, if $\mathcal{F}$ is a class of non-negative functionals on $L^p(\Omega) \times \mathcal{A}$, we say that the fundamental estimate holds uniformly in $\mathcal{F}$ if each element $F$ of $\mathcal{F}$ satisfies the fundamental estimate with $M$ depending only on $\varepsilon,A',A'',B,$ while $\varphi$ may depend also on $F,u,v$.
	\end{definition}
	The result in \cite[Theorem 19.1]{D} provides a wide class of integral functionals uniformly satisfying the fundamental estimate. In particular, the fundamental estimate holds uniformly in the class of all functionals of the form \eqref{functionals in Gamma convergence}. Thus for every $J^\lambda$, there exists a cut-off function $\varphi$ such that \eqref{cut-off.est} hold with $F=J^\lambda$ and a constant $M$ independent of $\lambda$.
\subsection{Uniform convergence of rescaled variational solutions}
Now we are ready to prove Theorem \ref{convergence of variational solutions}.
\begin{proof}[Proof of Theorem \ref{convergence of variational solutions}]
	For a fixed  $T > 0$, we can bound
	$\Omega_t(u^\lambda)$ by $B(0,R)$ for some $R >0$, for all $0 \leq t \leq
	T$ and $\lambda >0$ by Lemma~\ref{viscos free boundary bound}. We will show the convergence in  $Q_{\varepsilon}:= \left(B(0,R)\setminus
	\overline{B(0,\varepsilon)}\right) \times [0,T]$ for some $\varepsilon >0$.

We argue the same way as in the proof of \cite[Theorem 3.2]{PV}. Using the uniform bound on $u^\lambda, u^\lambda_t$ from Lemma~\ref{viscos free boundary bound} and the standard regularity estimates for an elliptic obstacle problem which hold uniformly in $\lambda$, we obtain a uniform H\"older estimate for $u^\lambda$. Then by the Arzel\`{a}-Ascoli theorem and a diagonalization argument, we can find a function $\bar u \in C((\Rn \setminus \{0\})
	\times [0, \infty))$ and a subsequence $\{u^{\lambda_k}\} \subset
	\{u^\lambda\}$ such that
	\begin{equation*}
	\begin{aligned}
	&u^{\lambda_k} \rightarrow \bar{u} \mbox{ locally uniformly on $(\Rn \setminus \{0\})
		\times [0, \infty)$} \mbox{ as } k
	\rightarrow \infty, \\
	&u^{\lambda_k}(\cdot,t) \rightarrow \bar{u}(\cdot,t) \mbox{ strongly in } H^1(\Omega_\varepsilon) \mbox{ for all } t\geq 0, \varepsilon > 0.
	\end{aligned}
	\end{equation*}

	In the rest of the proof we show that the function $\bar{u}$ solves the limit
  problem \eqref{Limit problem for variational solution},  whose uniqueness then implies the convergence of the full sequence. We start with quantifying the singularity at the origin.
\begin{lemma}
	\label{singularity of U}
	We have
	\begin{equation*}
	\lim_{|x|\rightarrow 0}\frac{\overline{u}(x,t)}{U_{C_*,L}(x,t)}=1.
	\end{equation*}
\end{lemma}
\begin{proof}
	Let $C_*$ as in Lemma~\ref{C*} and $F$ be the fundamental solution of \eqref{elliptic eq} as in Section~\ref{sec: sub and super sol}. Fix $\varepsilon>0$. By Lemma~\ref{C*}, there exists $a$ large enough such that
	\begin{equation}
	\label{estimate in P}
	\begin{aligned}
	&\left| \frac{P(x)}{F(x)}-C_*\right| < \frac{\varepsilon}{2}, &&\mbox{ in } \{|x|\geq a\}
	\end{aligned}
	\end{equation}
  and $K \subset \{|x| < a\}$.
	 In particular, \eqref{estimate in P} holds for every $x, |x|=a$.

	The set $\{|x|=a\}$ is a compact subset of $\mathbb{R}^n \setminus K$. Then by Theorem \ref{Near field limit Theorem}, there exists $t_0 >0$ such that for all $t \geq t_0$,
	\begin{equation*}
	\begin{aligned}
	&\left|\frac{v(x,t)}{F(x)}-\frac{P(x)}{F(x)}\right|< \frac{\varepsilon}{2}, && \mbox{ for all } x, |x|=a.
	\end{aligned}
	\end{equation*}
	By triangle inequality we have for all $t\geq t_0$, for all $x$ such that $|x| =a$,
	$$\displaystyle \left|\frac{v(x,t)}{F(x)}-C_*\right| < \varepsilon.$$

	Let $\Phi(x,t)$ be the fundamental solution of the parabolic equation
	\begin{equation}
	\label{parabolic equation}
	u_t-\mathcal{L}u=0.
	\end{equation}
	As shown in \cite{FriedmanP, Aronson}, such unique fundamental solution exists and satisfies
	\begin{equation}
	\label{parabolic fund.sol estimate}
	N^{-1}t^{-\frac{n}{2}}e^{-\frac{N|x|^2}{t}} \leq \Phi(x,t) \leq Nt^{-\frac{n}{2}}e^{-\frac{|x|^2}{Nt}}
	\end{equation}
	for some $N>0$.
	We consider $\theta_1,\theta_2$ as follows:
	\begin{align*}
    \theta_1(x,t)&:=\left[(C_*-\varepsilon)F(x)+\frac{c_2h(x)}{t}-c_3t^{\frac{2-n}{n}}\right]_+ \chi_{E}(x,t),\\
	\theta_2(x,t)&:= (C_*+\varepsilon)F(x) + C_2 \Phi(x,t),
	\end{align*}
	where $E$, $h(x)$ were defined as in Section \ref{subsol}. We will show that we can choose the coefficients such that $\theta_1$ is a subsolution and $\theta_2$ is a supersolution of \eqref{Stefan} in $\{|x|\geq a\}\times \{t\geq t_0\}$ for some $t_0$. Since we fix the first coefficient of $\theta_1$ and $\theta_2$, we need to check the initial conditions carefully.

  Note that on the set $\{|x|=a\}$, $\theta_1 \to (C_*-\varepsilon)F(x)$ and $\theta_2 \to (C_*+\varepsilon)F(x)$ uniformly as $t \to \infty$. Thus we can choose a large time $t_0$ such that $\theta_1 \leq v \leq \theta_2$ on $\{|x=a|\}\times \{t \geq t_0\}$. By \eqref{radius of E(t)}, we can choose $c_2$ large enough such that  $\supp \theta_1(\cdot,t_0) \subset \{x: (x, t_0) \in E\} \subset B_a(0)$  and then $\theta_1(\cdot, t_0) \leq v(\cdot,t_0)$ in $\{|x|\geq a\}$. Following Section \ref{subsol}, by choosing larger $c_2, t_0$ if necessary  and $c_3$ satisfying \eqref{subsolution coef condition 1}, \eqref{subsolution coef condition 2}, $\theta_1$ is a subsolution of \eqref{Stefan} in $\{|x|\geq a\} \times \{t \geq t_0\}$.

  Fix the time $t_0$ such that $\theta_1$ is a subsolution of \eqref{Stefan} in $\{|x|\geq a\} \times \{t\geq t_0\}$ as above. By \eqref{bound for fund.sol n>2} and \eqref{parabolic fund.sol estimate}, $\theta_2 >0$ in $\Rn$. Moreover, since $F(x)$ and $\Phi(x,t)$ are the fundamental solutions of \eqref{elliptic eq} and \eqref{parabolic equation} respectively, clearly $(\theta_2)_t - \mathcal{L}\theta_2 =0$ in $\Rn \setminus \{0\}$. If we choose $C_2$ large enough then $\theta_2(\cdot, t_0)> v(\cdot, t_0)$ and $\theta_2$ is a super solution of \eqref{Stefan} in $\{|x|\geq a\}\times \{t \geq t_0\}$.

	 By comparison principle, $\theta_1 \leq v \leq \theta_2$ in $\{|x|\geq a\} \times \{t\geq t_0\}$. Moreover, since $h(x)>0$ then $$\theta_1(x,t) \geq \tilde{\theta}_1(x,t):= \left[(C_*-\varepsilon)F(x)-c_3t^{\frac{2-n}{n}}\right]_+.$$ Therefore $\tilde{\theta}_1^\lambda \leq v^\lambda \leq \theta_2^\lambda$ for $\lambda$ is large enough.

	Noting that $\Phi^\lambda(x,t):=\lambda^{\frac{n-2}{n}}\Phi(\lambda^{\frac{1}{n}}x,\lambda t) \to 0$ uniformly as $\lambda \to \infty$ by \eqref{parabolic fund.sol estimate}, then by Lemma~\ref{barrier convergence of rescaled }, $\tilde{\theta}_1^\lambda, \theta_2^\lambda$ converge locally uniformly to $\theta_1^0, \theta_2^0$ of the form
	\begin{align*}
	\theta_1^0(x,t)&:=\left[(C_*-\varepsilon)F^0(x)-c_3t^{\frac{2-n}{n}}\right]_+ ,\\
	\theta_2^0(x,t)&:= (C_*+\varepsilon)F^0(x),
	\end{align*}
	where $F^0$ is the fundamental solution of $ -\mathcal{L}^0 u =0$, $\mathcal{L}^0$ is the limit of the operators $\mathcal{L}^\lambda$ as in Lemma~\ref{Asymptotic expansion fund.sol}.
	Applying the same method as in \cite{P1} we have
	\begin{equation}
	\label{order}
	\int_{0}^{t}\theta^0_1(x,s)\ds \leq \overline{u}(x,t) \leq \int_{0}^{t}\theta^0_2(x,s)\ds.
	\end{equation}
By Lemma~\ref{barrier integration in time lemma} we obtain
\begin{equation*}
 (C_*-\varepsilon)F^0(x)t-\frac{c_3n}{2}t^{\frac{2}{n}}+o(F^0(x)) \leq \overline{u}(x,t) \leq  (C_*+\varepsilon)F^0(x)t
\end{equation*}
as $|x| \to 0$.
	Dividing both sides of by $F^0(x)$ and taking the limit as $|x| \rightarrow 0$ we get
	\begin{equation*}
	(C_*-\varepsilon)t \leq \liminf_{|x|\rightarrow 0}\frac{\overline{u}(x,t)}{F^0(x)}\leq \limsup_{|x|\rightarrow 0}\frac{\overline{u}(x,t)}{F^0(x)}\leq (C_*+\varepsilon)t.
	\end{equation*}
	Since $\varepsilon>0$ is arbitrary, we have the correct singularity by sending $\varepsilon$ to $0$.
\end{proof}

Finally, we check that the limit function $\bar{u}$ satisfies the inequality and equality in  \eqref{Limit problem for variational solution}.
\begin{lemma}
	\label{limit function satisfies limit problem}
		For each $0 \leq t \leq T$, $\overline{w}=\overline{u}(\cdot,t)$ satisfies
	\begin{align}
	\label{ineq1}
		q(\overline{w},\phi)& \geq \left <-L,\phi\right >, &&\forall \phi \in W_1,\\
		\label{ineq2}
		q(\overline{w},\psi \overline{w})&= \left< -L, \psi \overline{w}\right>	,&&\forall \psi \in W_2,
	\end{align}
 where $L=\left< \frac{1}{g} \right>$ and $W_1,W_2$ were defined as in Section~\ref{sec: limit variational problem}.
\end{lemma}

\begin{proof}
	Fix $t \in [0,T]$ and take any $\phi \in W_1$. By continuity, we can choose $\phi$ with a compact support contained in $\Omega:= B(0,R) \setminus \overline{B(0,\varepsilon_0)}$ for some $0 < \varepsilon_0 < R$. Let $w^k(x):=u^{\lambda_k}(x,t)$ and $\overline{\varphi}:=\overline{w}+\phi \in H^1(\Rn)$. By Theorem \ref{gamma convergence}, there exists a sequence $\{\varphi^k\}$ that converges strongly in $L^2(\Omega)$ to $\overline{\varphi}$ such that
	\begin{equation}
	\label{convergence sequence}
	J^{\lambda_k}(\varphi^k, \Omega) \rightarrow J^0(\overline{\varphi},\Omega).
	\end{equation}
	We will show that we can modify $\varphi^k$ into $\tilde{\varphi}^k$ such that $\tilde{\varphi}^k \in \mathcal{K}^{\lambda_k}(t)$ and all the convergences are preserved.

	First, we see that $J^0(\bar{\varphi}, \Omega) < \infty$ since $\bar{\varphi} \in H^1(\Omega)$. By \eqref{convergence sequence}, $J^{\lambda_k}(\varphi^k, \Omega) < \infty$  and hence  $\varphi^k \in H^1(\Omega)$ when $k$ is large enough.

Next, we need to modify $\varphi^k$ so that the boundary condition on $K^{\lambda_k}$ is satisfied.
 Since $\overline{\varphi} \in H^1(\Omega)$, for every $\varepsilon>0$, there exists a compact set $A(\varepsilon) \subset \Omega$ such that $\supp \phi \subset A(\varepsilon)$ and
 \begin{equation}
 	\label{negliability}
 	\int_{\Omega \setminus A(\varepsilon)}|D\overline{\varphi}|^2\dx <\varepsilon.
 \end{equation}
 Let $A'(\varepsilon),A''(\varepsilon)$ such that $A(\varepsilon) \subset A'(\varepsilon) \Subset A''(\varepsilon) \Subset \Omega$ and $B(\varepsilon)=\Omega \setminus A(\varepsilon)$. By \cite[Theorem 19.1]{D}, the fundamental estimate \eqref{cut-off.est} holds uniformly in the class of all functionals of the form \eqref{functionals in Gamma convergence}. Thus there exists a constant $M\geq0$ independent of $\lambda_k$ and a sequence of cut-off functions $\xi^k_\varepsilon \in C^\infty_0(A''(\varepsilon)), 0 \leq \xi^k_\varepsilon \leq 1, \xi^k_\varepsilon = 1$ in a neighborhood of $\overline{A'(\varepsilon)}$ such that
 \begin{equation}
 	\label{phi.fund.est}
 	\begin{aligned}
 		J^{\lambda_k}(\xi^k_\varepsilon \varphi^k+(1-\xi^k_\varepsilon)(w^k+\phi),\Omega)\leq &(1+\varepsilon)(J^{\lambda_k}(\varphi^k,A''(\varepsilon))+J^{\lambda_k}(w^k+\phi, B(\varepsilon)))\\
 		&+\varepsilon(\|\varphi^k\|_{L^2(\Omega)}^2 + \|w^k+\phi\|_{L^2(\Omega)}^2+1)\\
 		&+M\|\varphi^k -w^k-\phi\|_{L^2(\Omega)}^2.
 	\end{aligned}
 \end{equation}
 Define
 \begin{equation*}
 \varphi_\varepsilon^k(x):=\left\{
 \begin{aligned}
&\xi^k_\varepsilon(x)\varphi^k(x)+(1-\xi^k_\varepsilon(x))(w^k(x)+\phi(x)) && \mbox{ if } x \in \Omega,\\
& w^k(x) && \mbox{ if } x \notin \Omega.
 \end{aligned}
 \right.
 \end{equation*}
Then $\varphi_\varepsilon^k \in H^1(\Rn), \|\varphi_\varepsilon^k-\bar{\varphi}\|_{L^2(\Omega)} \leq \|\varphi^k -\bar{\varphi}\|_{L^2(\Omega)} +\|w^k+\phi - \bar{\varphi}\|_{L^2(\Omega)} \to 0$ as $k \to \infty$ and $\varphi^k_\varepsilon-w^k$ has compact support in $\Omega$.

By ellipticity \eqref{ellipticity} we have
\begin{equation}
\label{gamma convergence ellipticity}
J^{\lambda_k}(w^k +\phi,B(\varepsilon)) \leq \beta\int_{B(\varepsilon)}|D(w^k+\phi)|^2\dx.
\end{equation}
In view of \eqref{negliability}, choose the sequence $\varepsilon_n:=\frac{1}{n}$ and denote $\varphi^k_n:=\varphi^k_{\varepsilon_n}$.
	By \eqref{phi.fund.est}, \eqref{gamma convergence ellipticity}, and the convergences $\varphi^k_n \to \overline{\varphi}$ in $L^2(\Omega)$ and $w^k \to \overline{w}$ in $H^1(\Omega)$ as $k \to \infty$, for each $n$ there exists $k_0(n)$ such that
	\begin{equation}
	\label{gamma convergence limsup}
	\left\{
	\begin{aligned}
	\|\varphi^k_n -\overline{\varphi}\|_{L^2(\Omega)} &\leq \min \left\{\frac{1}{n}, \frac{1}{Mn}\right\},\\
	J^{\lambda_k}(\varphi_n^k, \Omega) &\leq \left(1+\frac{1}{n}\right)\left(J^0(\overline{\varphi}, \Omega)+ \frac{\beta+1}{n}\right)+\frac{1}{n}\left(2\|\overline{\varphi}\|_{L^2(\Omega)}+\frac{1}{n}+1\right) + \frac{2}{n},
	\end{aligned}
	\right.
	\end{equation}
	for every $k \geq k_0(n)$. We can choose $k_0(n)$ such that $k_0$ is an increasing function of $n$ and $k_0(n) \to \infty$ as $n\to \infty$. We will form a new sequence $\{\hat{\varphi}^k\}$ from the class of sequences $\{\varphi^k_n\}$. The idea is that for each $k$, we will choose an appropriate $n(k)$ and set $\hat{\varphi}^k:=\varphi^k_{n(k)}$. We need to choose a suitable $n(k)$ such that $n(k) \to \infty$ and \eqref{gamma convergence limsup} holds for $\varphi^k_{n(k)}$ when $k$ is large enough. To this end we introduce an ``inverse" of $k$ as
	\begin{equation*}
	n(k):= \min \{j\in \mathbb{N}: k < k_0(j+1) \}.
	\end{equation*}
	$n(k)$ is well-defined, non-decreasing and tends to $\infty$ as $k\to \infty$.
	From the definition of $n(k)$ we see that if $k \geq k_0(2)$ then $n(k)\geq 2$ and $k_0(n(k)) \leq k < k_0(n(k)+1)$ (otherwise $n(k)$ is not the minimum). Thus by \eqref{gamma convergence limsup} and definition of $\hat{\varphi}^k$ we have for all $k \geq k_0(2)$,
	\begin{equation*}
	\left\{
	\begin{aligned}
	\|\hat{\varphi}^k -\overline{\varphi}\|_{L^2(\Omega)} &\leq \min \left\{\frac{1}{n(k)}, \frac{1}{Mn(k)}\right\},\\
	J^{\lambda_k}(\hat{\varphi}^k, \Omega)&=J^{\lambda_k}(\varphi^k_{n(k)}, \Omega) \\
	&\leq \left(1+\frac{1}{n(k)}\right)\left(J^0(\overline{\varphi}, \Omega)+ \frac{\beta+1}{n(k)}\right)\\
	&\quad+\frac{1}{n(k)}\left(2\|\overline{\varphi}\|_{L^2(\Omega)}+\frac{1}{n(k)}+1\right) + \frac{2}{n(k)}.
	\end{aligned}
	\right.
	\end{equation*}

	Sending $k \to \infty$ we get
	\begin{equation*}
	\left\{
	\begin{aligned}
	\lim_{k \to \infty}\|\hat{\varphi}^k - \bar{\varphi}\|_{L^2(\Omega)} &=0,\\
	\limsup_{k \to \infty} J^{\lambda_k}(\hat{\varphi}^k, \Omega) &\leq J^0(\bar{\varphi},\Omega).
	\end{aligned}
	\right.
	\end{equation*}
On the other hand, by Theorem \ref{gamma convergence},
 \begin{equation*}
 J^0(\bar{\varphi},\Omega) \leq \liminf_{k \to \infty}J^{\lambda_k}(\hat{\varphi}^k,\Omega)
 \end{equation*}
 and thus we can conclude that $\hat{\varphi}^k \to \bar{\varphi}$ strongly in $L^2(\Omega)$ and  $J^{\lambda_k}(\hat{\varphi}^k,\Omega) \to J^0(\bar{\varphi},\Omega)$. Moreover, by the definitions of $\varphi^k_\varepsilon, \hat{\varphi}^k$, we also have $\hat{\varphi}^k \in H^1(\Omega)$ and $\hat{\varphi}^k -w^k$ has compact support in $\Omega$.

 Now set
 $\tilde{\varphi}^k:=|\hat{\varphi}^k|$.
Then $\tilde{\varphi}^k \in H^1(\Omega), \tilde{\varphi}^k \geq 0, \tilde{\varphi}^k = w^k$ in $\Omega^\compl \supset K^{\lambda_k}$ for $k$ large enough, and thus $\tilde{\varphi}^k \in \mathcal{K}^{\lambda_k}(t)$ for $k$ large enough. Moreover, following the argument in the proof of \cite[Lemma~4.5]{K3},
$\tilde{\varphi}^k \to \bar{\varphi}$ in $L^2(\Omega)$ and $J^{\lambda_k}(\tilde{\varphi}^k,\Omega) \to J^0(\bar{\varphi},\Omega)$.

	Since $w^k, \tilde{\varphi}^k \in \mathcal{K}^{\lambda_k}(t)$ and $\supp(\tilde{\varphi}^k-w^k) \subset \Omega $, by \eqref{rescaled Variational problem} and integration by parts formula we have
	\begin{equation*}
	a^{\lambda_k}_{\Omega}(w^k, \tilde{\varphi}^k -w^k) \geq -\lambda_k^{\frac{2-n}{n}}\left<u^{\lambda_k}_t, \tilde{\varphi}^k -w^k\right>_{\Omega} + \left<-\frac{1}{g^{\lambda_k}},\tilde{\varphi}^k -w^k\right>_{\Omega}.
	\end{equation*}
The inequality $a^{\lambda_k}(u,v-u) \leq \frac{1}{2}J^{\lambda_k}(v)- \frac{1}{2}J^{\lambda_k}(u)$ for any $u,v$ implies
	\begin{equation*}
	\frac{1}{2} J^{\lambda_k}(\tilde{\varphi}^k, \Omega) \geq \frac{1}{2} J^{\lambda_k}(w^k, \Omega)- \lambda_k^{\frac{2-n}{n}}\left<u^{\lambda_k}_t, \phi^k\right>_{\Omega} + \left<-\frac{1}{g^{\lambda_k}},\phi^k\right>_{\Omega},
	\end{equation*}
	where $\phi^k:= \tilde{\varphi}^k - w^k \to \phi$ in $L^2(\Omega)$. Taking $\liminf$ as $k \rightarrow \infty$ and using the fact that $u^{\lambda_k}_t$ is bounded give
	\begin{equation}
	\label{gamma convergence inequality}
	\frac{1}{2} J^0(\overline{\varphi}, \Omega) \geq \frac{1}{2} J^0(\overline{w}, \Omega) + \left<-L,\phi\right>_{\Omega}.
	\end{equation}
	This holds for any $\phi \in W_1$ and therefore also for $\delta\phi$, where $0 <\delta <1$.
Replacing $\phi$ in \eqref{gamma convergence inequality} by $\delta\phi$ we have
\begin{align*}
\frac{1}{2}J^0(\bar{w}+\delta\phi,\Omega) &\geq \frac{1}{2}J^0(\bar{w},\Omega)+\left<-L, \delta\phi\right>\\
\Leftrightarrow  \frac{1}{2}\left[J^0(\bar{w},\Omega)+2\delta q_\Omega(\bar{w},\phi)+\delta^2J^0(\phi)\right] &\geq \frac{1}{2}J^0(\bar{w},\Omega)+\left<-L, \delta\phi\right>.
\end{align*}
Dividing both sides by $\delta$ and sending $\delta \to 0$ we obtain
	\begin{equation*}
	q_\Omega(\overline{w},\phi) \geq \left<-L, \phi\right>_\Omega.
	\end{equation*}
Since $\supp \phi \in \Omega$, we conclude that \eqref{ineq1} holds in $\Rn. $

	Now take $\psi \in W_2$. As above, we assume that $\psi$ has a compact support contained in  $\Omega$, and  without loss of generality we can also assume that $0 \leq \psi \leq 1, \psi =0$ on $B_\varepsilon(0)$ (otherwise consider $\frac{\psi}{\max_{\Rn}\psi}$ instead). Since $\psi \in W_2$ then $\psi \overline{w} \in W_1$ and \eqref{ineq1} holds for $\psi\bar{w}$, we have $q(\overline{w}, \psi \overline{w}) \geq \left<-L, \psi \overline{w}\right>$. For the reverse inequality, define $\overline{\varphi}:= (1-\psi)\overline{w} \in H^1(\Omega)$. Arguing as before, we can choose $\tilde{\varphi}^k \in \mathcal{K}^{\lambda_k}(t)$ such that $\tilde{\varphi}^k \to \overline{\varphi} \mbox{ in } L^2(\Omega), J^{\lambda_k}(\tilde{\varphi}^k, \Omega) \to J^0(\overline{\varphi}, \Omega)$. Again, since $w^k, \tilde{\varphi}^k \in \mathcal{K}^{\lambda_k}(t)$, by \eqref{rescaled Variational problem} and the inequality $a^{\lambda_k}(u,v-u) \leq \frac{1}{2}J^{\lambda_k}(v)- \frac{1}{2}J^{\lambda_k}(u)$ for any $u,v$ we have
	\begin{equation*}
    \frac{1}{2} J^{\lambda_k}(\tilde{\varphi}^k, \Omega) \geq \frac{1}{2} J^{\lambda_k}(w^k, \Omega)- \lambda_k^{\frac{2-n}{n}}\left<u^{\lambda_k}_t, \tilde{\varphi}^k-w^k\right>_{\Omega} + \left<-\frac{1}{g^{\lambda_k}},\tilde{\varphi}^k -w^k\right>_{\Omega}.
	\end{equation*}
	Taking $\liminf$ as $k \to \infty$ and arguing the same as in the proof of \eqref{ineq1} we get
	\begin{equation*}
	\begin{aligned}
	&&q_{\Omega}(\overline{w}, \overline{\varphi}- \overline{w}) &\geq \left<-L, \overline{\varphi}- \overline{w}\right>_\Omega\\
	&\Leftrightarrow& -q_\Omega(\overline{w}, \psi \overline{w}) &\geq -\left<-L, \psi \overline{w} \right>_\Omega\\
	&\Leftrightarrow& q_\Omega(\overline{w}, \psi \overline{w}) &\leq \left<-L, \psi \overline{w} \right>_\Omega.
	\end{aligned}
	\end{equation*}
	Thus $ q(\overline{w}, \psi \overline{w}) = \left<-L, \psi \overline{w} \right>$ for every $\psi \in W_2$.
	\end{proof}
	This completes the proof of Theorem \ref{convergence of variational solutions}.
	\end{proof}
	\section{Uniform convergence of rescaled viscosity solutions and free boundaries}
   In this final section, we establish the convergence of the rescaled viscosity solutions $v^\lambda$ of the Stefan problem (\ref{Stefan})  and their free boundaries. The proof is based on viscosity arguments showing that the half-relaxed limits of $v^\lambda$ in $\{|x| \neq 0, t\geq 0\}$ defined as
 \begin{align*}
   & v^*(x,t)= \limsup_{(y,s),\lambda \rightarrow (x,t), \infty} v^\lambda(y,s), & v_*(x,t)= \liminf_{(y,s),\lambda \rightarrow (x,t), \infty} v^\lambda(y,s)
 \end{align*}
 coincide and are the viscosity solution of the limit problem with a point source.
 We have the following result, which is similar to \cite[Theorem 4.2]{PV}.
 \begin{theorem}
 	\label{convergence of rescaled viscosity solution}
 	Let $n\geq 3$ and $V=V_{C_*,L}$ be the solution of Hele-Shaw problem with a point source \eqref{Hele_shaw point source problem} with the constant $C_*$ from Lemma~\ref{C*} and $L=\left<\frac{1}{g}\right>$ as in Lemma~\ref{media}. The rescaled viscosity solution $v^\lambda$ of the Stefan problem (\ref{Stefan}) converges
 	locally uniformly to $V = V_{C_*, \ang{\frac{1}{g}}}$ in $(\mathbb{R}^n \setminus \{0\}) \times [0,\infty)$ as $\lambda \rightarrow \infty$ and
 	$$v_*=v^*=V.$$
 	Moreover, the rescaled free boundary $\{\Gamma(v^\lambda)\}_{\lambda}$ converges to $\Gamma(V)$ locally uniformly with respect to the Hausdorff distance.
 \end{theorem}

 All the viscosity arguments used in \cite[Section 4]{PV} can be applied in our anisotropic case with some minor adaptations. Therefore, we will omit some of the proofs and refer to \cite{K2,K3,P1,PV} for more details. Let us give a brief review of the ideas in the spirit of \cite[Section 4]{PV} as follows.
 \begin{enumerate}
 	\item We first prove the convergence of the rescaled viscosity solution and its free boundary under the condition \eqref{initial data}.
 	\begin{itemize}
 		\item By the regularity of the initial data $v_0$ as in \eqref{initial data}, we deduce a weak monotonicity of the solution $v$.
 		\item Using the weak monotonicity and pointwise comparison principle arguments, we then show the convergence for regular initial data.
 	\end{itemize}
\item For general initial data, we will find regular upper and lower approximations of the initial data satisfying \eqref{initial data} and use the comparison principle together with the  uniqueness of the limit solution to reach the conclusion.
 \end{enumerate}

 We will state the necessary results here with remarks on the adaptations for the anisotropic case.
 \subsection{Some necessary technical results}
 First, we have the correct singularity of $v^*$ and $v_*$ at the origin, which can be established similarly to the proof of Lemma~\ref{singularity of U}.
 \begin{lemma}[cf. {\cite[Lemma~4.3]{PV}}, $v^*$ and $v_*$ behave as $V$ at the origin]
 	\label{boundary condition for limit problem}
 	The functions $ v^*, v_*$ have a singularity at $0$ with
 	\begin{align}
 	\label{singularity of V}
 	&\lim_{|x| \rightarrow 0+}\frac{v_*(x,t)}{V(x,t)}=1, & \lim_{|x| \rightarrow 0+}\frac{v^*(x,t)}{V(x,t)}=1, \quad \mbox{ for } t>0.
 	\end{align}
 \end{lemma}
 \begin{proof}
 	Argue as in the proof of Lemma~\ref{singularity of U}.
 \end{proof}
We will also make use of an uniform estimate on $u^\lambda$ and the convergence of boundary points deduced from the convergence of variational solutions.
 \begin{lemma}[cf. {\cite[Lemma~3.1]{K2}}]
 	\label{positive sup rescaling}
 	There exists constant $C > 0$ independent of $\lambda$ such that for every $x_0 \in \overline{\Omega_{t_0}(u^\lambda)}$ and $B_r(x_0) \cap \Omega_0^{\lambda} = \emptyset$ for some $r$, we have
 	\begin{equation*}
 	\sup_{x \in \overline{B_r(x_0)}}u^\lambda(x,t_0)>C r^2
 	\end{equation*}
 	 for every $\lambda$.
 \end{lemma}
 \begin{proof}
 	We will prove the statement for $x_0 \in \Omega_{t_0}(u^\lambda)$ first, the results then follows by continuity of $u^\lambda$. Since $B_r(x_0) \cap \Omega_0^{\lambda}= \emptyset$ then $u^\lambda$ satisfies
 	\begin{equation*}
 	\lambda^{\frac{2-n}{n}}u^\lambda_t-\mathcal{L}^\lambda u^\lambda= -\frac{1}{g^\lambda} \mbox{ in } \{u^\lambda>0\}\cap (B_r(x_0) \times \{t=t_0\}).
 	\end{equation*}
 	Since $u^\lambda_t \geq 0$ and $-\frac{1}{g}\leq -\frac{1}{M}$ then  $-\mathcal{L}^\lambda u^\lambda \leq -\frac{1}{M}=:-C_0$  in  $\{u^\lambda>0\} \cap (B_r(x_0) \times\{t=t_0\})$.

 	Define
 	\begin{equation*}
 	w^\lambda(x)=u^\lambda(x,t_0)-\frac{C_0}{n}h^\lambda(x-x_0)
 	\end{equation*}
 	where $h^\lambda(x)$ is the barrier with quadratic growth corresponding to elliptic operator $\mathcal{L}^\lambda$ introduced in Section \ref{subsol}. We have $\{w^\lambda>0\} \cap B_r(x_0) \subset \{u^{\lambda}>0\} \cap \{t=t_0\}$ and therefore, for all $\lambda$, $$-\mathcal{L}^\lambda w^\lambda \leq 0 \mbox{ in } \{w^\lambda>0\} \cap B_r(x_0).$$

 	We see that $w^\lambda(x_0)>0$. Hence the maximum of $w^\lambda$ in $\overline{B_r(x_0)}$ is positive and by the maximum principle, $w^\lambda$ attains the maximum on the boundary $\{w^\lambda>0\} \cap \partial B_r(x_0)$ and therefore
 	\begin{equation*}
 	\sup_{\overline{B_r(x_0)}} u^\lambda(x,t_0) \geq \sup_{|x-x_0|=r} u^\lambda(x,t_0) > \inf_{|x-x_0|=r}\frac{C_0}{n} \, h^\lambda(x-x_0).
 	\end{equation*}
 	By the quadratic growth of $h^\lambda$, where the coefficients on the growth rate only depend on the elliptic constants, we have
 	\begin{equation*}
 	\sup_{\overline{B_r(x_0)}} u^\lambda(x,t_0) \geq Cr^2,
 	\end{equation*}
 	for some constant $C$ which does not depend on $\lambda$.
 \end{proof}
  \begin{lemma}[cf. {\cite[Lemma~5.4]{K3}}]
  	\label{limit of sequence in 0-level set of ulambda k}
  	Suppose that $(x_k,t_k) \in \{u^{\lambda_k}=0\}$ and $(x_k,t_k,\lambda_k) \rightarrow (x_0,t_0,\infty)$. Let $U=U_{C_*,L}$ be the limit function as in Theorem~\ref{convergence of variational solutions}. Then:\begin{enumerate}[label=\alph*)]
  		\item $U(x_0,t_0)=0$,
  		\item If $x_k \in \Gamma_{t_k}(u^{\lambda_k})$ then $x_0  \in \Gamma
  		_{t_0}(U)$,
  	\end{enumerate}
  \end{lemma}
  \begin{proof}
  	See the proof of \cite[Lemma~5.4]{K3}.
  \end{proof}

A weak monotonicity in time of the solution of the Stefan problem \eqref{Stefan} is given by the following lemma.
 \begin{lemma}[cf. {\cite[Lemma~4.7, Lemma~4.8]{PV}}, Weak monotonicity]
 	\label{weak monotonicity}Let $u$ be the solution of the variational problem \eqref{Variational problem}, and $v$ be the associated viscosity solution of the Stefan problem.
 	Suppose that $v_0$ satisfies \eqref{initial data}. Then there exist $C \geq 1$ independent of $x$ and $t$ such that
 	\begin{equation}
 	\label{monotonicity condition}
 	v_0(x) \leq C v(x, t) \mbox{ and } u(x,t)\leq C t v(x,t) \mbox{ in } \mathbb{R}^n \setminus K \times [0,\infty).
 	\end{equation}
 \end{lemma}
 \begin{proof}
 Following the same arguments as in \cite[Lemma~4.7, Lemma~4.8] {PV}, we obtain \eqref{monotonicity condition} simply by using elliptic operator $\mathcal{L}$ instead of the Laplace operator.
 	\end{proof}

Lemma~\ref{positive sup rescaling} and Lemma~\ref{weak monotonicity} automatically give us a crucial uniform lower estimate on $v^\lambda$ and allow us to show the relationship between $v_*,v^*$ and $V$.
 \begin{cor}
 	\label{monotoniciy for v}
 	There exists a constant $C_1=C_1(n,M)$ such that if $(x_0,t_0) \in \Omega(v^\lambda)$ and $B_r(x_0) \cap \Omega_0^\lambda = \emptyset$, we have
 	\begin{equation*}
 	\sup_{B_r(x_0)}v^\lambda(x,t_0) \geq \frac{C_1r^2}{t_0}.
 	\end{equation*}
 \end{cor}

 \begin{lemma}
 	\label{relationship v^*,v_*,V}
 	Let $v$ be the viscosity solution of \eqref{Stefan} and $v^\lambda$ be its rescaling. Then the following statements hold.
 	\begin{enumerate}[label=\roman*)]
 		\item \label{i}$v^*(\cdot,t)$ is a subsolution of \eqref{elliptic hom eq} in $\Rn \setminus \{0\}$ and $v_{*}(\cdot,t)$ is a supersolution of \eqref{elliptic hom eq} in $\Omega_t(v_*) \setminus \{0\}$ in viscosity sense.
 		\item \label{ii}$\Omega(V) \subset \Omega(v_*)$ and in particular $v_* \geq V$.
 		\item \label{iii}$\Gamma(v^*) \subset \Gamma(V)$.
 	\end{enumerate}
 \end{lemma}
 \begin{proof}
 i) follows from standard viscosity arguments with noting that we can take a sequence of test functions for rescaled elliptic equation that converges to the test function for \eqref{elliptic hom eq} by classical homogenization results.

 ii) See \cite[Lemma~5.5]{K3}, the conclusion holds by i), Lemma~\ref{boundary condition for limit problem} and Lemma~\ref{weak monotonicity}.

 iii) See \cite[Lemma~5.6 ii]{K3}.
 \end{proof}

 \subsection*{Proof of Theorem \ref{convergence of rescaled viscosity solution}}
 \begin{proof} We follow the proof of \cite[Theorem 4.2]{PV}; see \cite{PV} for more details.

 	\textit{\textbf{Step 1}}. We first show the convergence results for the problem with the initial data satisfying (\ref{initial data}) using the weak monotonicity in time of the solution, Lemma~\ref{weak monotonicity}, and its consequences.

 	By Lemma~\ref{relationship v^*,v_*,V}, the correct singularity of $v^*$ from Lemma~\ref{boundary condition for limit problem} and the comparison principle for elliptic equation \eqref{elliptic hom eq} we have
 	$$V(x,t) \leq v_*(x,t) \leq v^*(x,t) \leq V_{C_*+\varepsilon, \left<\frac{1}{g}\right>}(x,t).$$
 	Let $\varepsilon \to 0$ we obtain $v_*=v^*= V$  by continuity and in particular, $\Gamma(v_*)= \Gamma(v^*)= \Gamma(V)$.

 	Now we show the locally uniform convergence of the free boundaries with respect to the Hausdorff distance. To simplify the notation, we fix $0 < t_1< t_2$ and define
 	\begin{align*}
    & \Gamma^\lambda:= \Gamma(v^\lambda) \cap \{t_1 \leq t \leq t_2\}, & \Gamma^\infty:= \Gamma(V) \cap \{t_1 \leq t \leq t_2\}.
 	\end{align*}
 	The result will follow if we show that for all $\delta >0$, there exists $\lambda_0>0$ such that for all $\lambda \geq \lambda_0$,
 	\begin{equation}
 	\label{1st Theorem 4.2}
  \begin{aligned}
 	\operatorname{dist}((x_0, t_0), \Gamma^\infty) &< \delta \text{ for all  $(x_0, t_0) \in \Gamma^\lambda$},  \mbox{ and }\\
  \operatorname{dist}((x_0, t_0), \Gamma^\lambda) &< \delta \text{ for all $(x_0, t_0) \in \Gamma^\infty$}.
  \end{aligned}
 	\end{equation}
 A contradiction argument as in the proofs of \cite[Theorem~7.1]{P1} and \cite[Theorem~4.2]{PV}, using Lemma~\ref{limit of sequence in 0-level set of ulambda k} above yields the existence of $\lambda_0$ for the first inequality. Hence the main task now is to show the existence of $\lambda_0$ for the second inequality in (\ref{1st Theorem 4.2}). Note that we only need to show this pointwise. The result then follows from the compactness of $\Gamma^\infty$. Suppose that there exists $\delta >0, (x_0,t_0) \in \Gamma^\infty$ and
 	$\{\lambda_k\}, \lambda_k \rightarrow \infty$, such that $\displaystyle \mbox{dist}((x_0,t_0),
 	\Gamma^{\lambda_k}) \geq \frac{\delta}{2}$ for all $k$. Then there exists $r >0$ such that  after passing to a subsequence if necessary, we can assume that $D_r(x_0,t_0):= B(x_0, r) \times [t_0-r,t_0+r]$ satisfies either
 	\begin{equation}
 	\label{4nd Theorem 4.2}
 	D_r(x_0,t_0) \subset \{v^{\lambda_k}=0\}, \quad \text{ for all } k
 	\end{equation}
 	or,
 	\begin{equation}
 	\label{3rd Theorem 4.2}
 	D_r(x_0,t_0) \subset \{v^{\lambda_{k}}>0\}, \quad \text{ for all } k.
 	\end{equation}
 	But (\ref{4nd Theorem 4.2}) clearly implies $V=v_* =0$  in $D_r(x_0,t_0)$, contradicting $(x_0,t_0) \in \Gamma^\infty$. Thus we assume (\ref{3rd Theorem 4.2}). Following \cite{PV}, to handle Harnack's inequality for a parabolic equation that becomes elliptic in the limit, we rescale time as
 	\begin{equation*}
 	w^k(x,t):=v^{\lambda_k}(x,\lambda_k^{\frac{2-n}{n}} t).
 	\end{equation*} Then $w^k>0$ in $D_r^w(x_0,t_0):=B(x_0,r) \times
 	[\lambda_k^{\frac{n-2}{n}}(t_0-r),\lambda_k^{\frac{n-2}{n}}(t_0+r)]$ and $w^k$ satisfies $w^k_t -\mathcal{L}^\lambda w^k=0$
 	in $D_r^w(x_0,t_0)$. Since $\lambda_k^{\frac{n-2}{n}} \to \infty$ as $k \to \infty$ then for any fixed $\tau>0$, there exists $\lambda_0$ such that $\tau <
 	\lambda_k^{\frac{n-2}{n}} \tfrac r4$ for all $\lambda_k \geq \lambda_0$. Now applying Harnack's inequality for the parabolic equation $w^k_t -\mathcal{L}^\lambda w^k=0$ we have for a fixed $\tau > 0$, there exists a constant $C_1 >
 	0$ such that for each $t \in [t_0-\frac r2, t_0+\frac r2]$ and all $\lambda_k$ such that $\tau <
 	\lambda_k^{\frac{n-2}{n}} \tfrac r4$ we have
 	\begin{equation*}
 	\sup_{B\left(x_0,\tfrac{r}{2}\right)} w^k(\cdot, \lambda_k^{\frac{n-2}{n}}t- \tau) \leq C_1 \inf_{B\left(x_0,\tfrac{r}{2}\right)}w^k(\cdot, \lambda_k^{\frac{n-2}{n}}t).
 	\end{equation*}
 As noted in \cite[Theorem 4.2]{PV}, for the isotropic case, the constant $C_1$ of Harnack's inequality  can be taken not depending on $\lambda_k$. For the anisotropic case, this constant also depends on the elliptic constants of operator $\mathcal{L}^{\lambda_k}$. However the rescaling of the operator does not change the elliptic constants. Thus $C_1$ can be taken independent of $\lambda_k$.
 	By this inequality and Corollary~\ref{monotoniciy for v} we have
 	\begin{equation*}
 	\frac{C_2r^2}{t- \lambda_k^{\frac{2-n}{n}}\tau} \leq   \sup_{B\left(x_0,\tfrac{r}{2}\right)} v^{\lambda_k}(\cdot, t- \lambda_k^{\frac{2-n}{n}}\tau) \leq C_1 \inf_{B\left(x_0,\tfrac{r}{2}\right)}v^{\lambda_k}(\cdot, t)
 	\end{equation*}
 	for all $t \in [t_0 - \frac r2, t_0 + \frac r2]$, $\lambda_k \geq \lambda_0$ large enough, where
 	$C_2$ only depends on $n, M$. In the limit $\lambda_k \rightarrow \infty$, the
 	uniform convergence of $\{v^{\lambda_k}\}$ to $V$ implies $V>0$ in $B(x_0,\frac r2) \times [t_0-\frac
 	r2,t_0+\frac r2]$, which contradicts the assumption $(x_0,t_0) \in \Gamma^\infty \subset \Gamma(V)$. This concludes the proof of Theorem~\ref{convergence of rescaled viscosity solution} when
 	(\ref{monotonicity condition}) holds.

 	\textit{\textbf{Step 2}}. For general initial data, arguing as in step 2 of the proof of \cite[Theorem 4.2]{PV}, we are able to find upper and lower bounds for the initial
 	data for which (\ref{monotonicity condition}) holds. The comparison principle for viscosity solution of the Stefan problem \eqref{Stefan} then yields the convergence since the limit function $V$ is unique and does not depend on the initial data.

 \end{proof}

 \section*{Acknowledgments}
 The first author was partially supported by JSPS KAKENHI Grants No. 26800068 (Wakate B) and No. 18K13440 (Wakate). This work
 is a part of doctoral research of the second author. The second author would like to thank her
 Ph.D.
 supervisor Professor Seiro Omata for his invaluable support and advice.



\end{document}